\documentclass[12pt,reqno]{amsart}
\usepackage[utf8]{inputenc}
\usepackage[T1]{fontenc}
\usepackage{lmodern}
\usepackage[english]{babel}
\usepackage{tikz,tikz-cd}
\usepackage{graphicx}
\usepackage{amsmath,a4wide}
\usepackage{xfrac}
\usepackage{esint}
\usepackage{stmaryrd,mathrsfs,bm,amsthm,mathtools,yfonts,amssymb,color,braket,booktabs,graphicx,graphics,amsfonts}
\usepackage{latexsym,microtype,indentfirst,hyperref}

\usepackage{xcolor}
\usepackage{courier}
\usepackage{constants}

\newtheorem{theorem}{Theorem}[section]

\newtheorem{lemma}[theorem]{Lemma}
\newtheorem{proposition}[theorem]{Proposition}

\theoremstyle{definition}

\newtheorem{remark}[theorem]{Remark}

\newtheorem{assumption}[theorem]{Assumption}

\numberwithin{equation}{section}
\numberwithin{subsection}{section}

\newcommand{\R}{\mathbb{R}} 

\DeclareMathOperator*{\esssup}{ess\,sup}

\newconstantfamily{eps}{
symbol=\varepsilon}
\newconstantfamily{con}{
symbol=c}
\title[Existence of curvature flow with forcing]{Existence of curvature flow with forcing \\ in a critical Sobolev space}
\author[Y. Liu]{Yuning Liu}
\address{NYU Shanghai, 567 Yangsi W road, Pudong, Shanghai 200126, China, and NYU-ECNU Institute of
Mathematical Sciences at NYU Shanghai, 3663 Zhongshan Road North, Shanghai, 200062, China}
\email{yl67@nyu.edu}
\author[Y. Tonegawa]{Yoshihiro Tonegawa}
\address{Department of Mathematics, Institute of Science Tokyo, 2-12-1 Ookayama, Meguro-ku, Tokyo 152-8551, Japan}
\email{tonegawa@math.titech.ac.jp}
\thanks{ Yoshihiro Tonegawa acknowledges the support of the JSPS Grant-in-aid for scientific research 23H00085.
 Yuning Liu  acknowledges the support of NSF of China under Grant 12571224.}
\date{\today}

\begin{document}
\begin{abstract}
  Suppose that a closed 1-rectifiable set $\Gamma_0\subset\R^2$ of finite 1-dimensional Hausdorff measure and
  a vector field $u$ in a dimensionally critical Sobolev space are given. 
It is proved that, starting from $\Gamma_0$, there exists a non-trivial flow
of curves with the normal velocity given by the sum of the curvature and the given vector field $u$. The motion law is satisfied in the sense of 
  Brakke and the flow exists through singularities. 
\end{abstract}

\medskip
\maketitle

\section{Introduction}
The present paper studies a general existence problem related to the mean curvature flow with a critical driving force, the meaning of ``critical'' being explained in the next paragraph. 
Given an $n$-dimensional submanifold $\Gamma_0\subset\mathbb R^{n+1}$
and a vector field $u(x,t):\mathbb R^{n+1}\times[0,\infty)\rightarrow\mathbb R^{n+1}$, consider
the existence problem of a family $\Gamma(t)\subset\mathbb R^{n+1}$ 
such that $\Gamma(0)=\Gamma_0$ and $\Gamma(t)$ moves by the motion law of
\begin{equation}\label{prob}
    v(x,t)=h(x,t)+u(x,t)^\perp.
\end{equation}
Here, $v(x,t)$ is the normal velocity vector of $\Gamma(t)$
at $x$, $h(x,t)$ is the mean curvature vector of $\Gamma(t)$ at $x$, and $u(x,t)^\perp$ is the projection of $u(x,t)$ to $({\rm Tan}_x \Gamma(t))^{\perp}$. One may view the problem 
as the mean curvature flow with a given background flow field $u$ as a forcing term, and the short-time existence of 
unique classical solution $\Gamma(t)$ satisfying \eqref{prob} can be 
proved if $\Gamma_0$ and $u$ are sufficiently smooth. 
When the given vector $u$ is not smooth, it is interesting to study the minimal requirement for the existence of solution, where the meaning of solution to \eqref{prob} itself can be an essential part of the study. Though the motivation
of the problem with a low-regularity assumption on $u$ may be considered as purely theoretical, 
there is a coupled two-phase flow problem of Navier-Stokes
equation proposed as a sharp interface limit of the Navier-Stokes/Allen-Cahn equation (see \cite{MR4645674,liu2012two,LiuShen}
and the references therein) where the presence of interface $\Gamma(t)$ influences the incompressible flow $u$ itself
via the surface tension and $\Gamma(t)$ moves by the motion law \eqref{prob} (when $\sigma_2=1$ below). 
We briefly mention this problem, where
$u(x,t)$, $\Pi(x,t)$ (pressure) and $\Gamma(t)$ satisfy in the 
distributional sense
\begin{equation}\label{NSeq}
    \begin{array}{ll}
    u_t+u\cdot\nabla u =\Delta u-\nabla \Pi+\sigma_1 h\mathcal H^n\lfloor_{\Gamma(t)}& \mbox{ on }\mathbb R^{n+1}\times (0,\infty),\\
    \quad \nabla\cdot u=0 &\mbox{ on }\mathbb R^{n+1}\times (0,\infty),\\
    \quad v=\sigma_2 h+u^\perp&\mbox{ on }\Gamma(t) \mbox{ for $t$>0}.
    \end{array}
\end{equation}
Here $\sigma_1,\sigma_2$ are positive constants and $ h\,\mathcal H^n\lfloor_{\Gamma(t)}$
is an $n$-dimensional surface measure on $\Gamma(t)$ multiplied by
the mean curvature vector $h$
of $\Gamma(t)$. This system has a natural energy law of 
\begin{equation}
    \frac{d}{dt}\Big(\int_{\mathbb R^{n+1}}\frac{|u|^2}{2}\,dx
    +\sigma_1\mathcal H^{n}(\Gamma(t))\Big)=-\int_{\mathbb R^{n+1}}|\nabla u|^2\,dx-\sigma_1\sigma_2\int_{\Gamma(t)}|h|^2\,d\mathcal H^n
\end{equation}
so that the sum of the kinetic energy of the fluid and the $n$-dimensional surface area of interface $\Gamma(t)$ decreases in time. 
In the present paper, we emphasize that we take $u$ to be a {\it given datum}, on the other hand, leaving the application to the coupled problem for the future. 

The main result of the present paper is a global-in-time 
existence of solution for the initial value problem of \eqref{prob} in the case of $n=1$ under the assumption that
the given vector field
$u$ satisfies 
\begin{equation}\label{NS}
    u\in L^{\infty}_{loc}([0,\infty);L^2(\mathbb R^2))
    \cap L^2_{loc}([0,\infty);W^{1,2}(\mathbb R^2)).
\end{equation}
Here $W^{1,2}(\R^2)$ is the
  Sobolev space with $L^2$-integrable weak first derivatives. 
Moreover, for all $T>0$, writing 
\begin{equation}\label{NS1}
    \Cl[con]{c1}(u,T):=\Big(\esssup_{t\in[0,T]}\int_{\mathbb R^2}|u(x,t)|^2\,dx\Big)
    \Big(\int_{0}^T\int_{\mathbb R^2}|\nabla u(x,t)|^2\,dxdt\Big),
\end{equation}
the solution satisfies
\begin{equation}\label{NS2}
\mathcal H^1(\Gamma(T))\leq \mathcal H^1(\Gamma_0)\exp(\Cl[con]{MZ}\, \Cr{c1}(u,T))
\end{equation}
for all $T>0$ with an absolute constant $\Cr{MZ}>0$. Here
$\mathcal H^1(A)$ for $A\subset\mathbb R^2$ is the one-dimensional Hausdorff measure of $A$. The solution $\Gamma(t)$
is guaranteed to be non-trivial at least for some positive time interval and may go through various occurrences of singularities. The functional space 
\eqref{NS} 
is a natural class to solve 
the above mentioned problem \eqref{NSeq}, but
more interestingly, the quantities appearing in \eqref{NS1} 
are invariant under the natural parabolic scaling (when $T=\infty$ and note that $u$ has the same scaling as $x/t$, see \cite{Caffarelli:1982aa}). 
As the most relevant results, the second-named author and collaborators showed (\cite{LiuSato2010,takasao2016existence}) such existence
results for general $n\geq 1$ if the given vector field $u$ satisfies 
\begin{equation}\label{subc}
    \int_0^T\left(\int_{\mathbb R^{n+1}}|u(x,t)|^p+|\nabla u(x,t)|^p\,dx\right)^{q/p}dt<\infty\mbox{ with } p>\frac{(n+1)q}{2(q-1)},\,q>2
\end{equation}
and $p\geq \frac43$ in addition if $n=1$. These conditions on $p,q$ imply that, under the
parabolic scaling, the effect of $u$ in small scale
becomes negligible and $u$ may be regarded as a perturbation to the mean curvature flow. More precisely, for
$\lambda>0$, let $\tilde x= \lambda^{-1} x$, $\tilde t=\lambda^{-2} t$
and $\tilde u(\tilde x,\tilde t)=\lambda u(x,t)$ (which should behave like $\tilde x/\tilde t$). Then the
change of variables shows
\begin{equation}
\Big(\int\Big(\int |\nabla u(x,t)|^p\,dx\Big)^{\frac{q}{p}}
    dt\Big)^{\frac{1}{q}}=\Big(\int
    \Big(\int|\nabla \tilde u(\tilde x,\tilde t)|^p\,d\tilde x\Big)^{\frac{q}{p}} d\tilde t\Big)^{\frac{1}{q}}\lambda^{-2+\frac{n+1}{p}+\frac{2}{q}},
\end{equation}
and the condition on the exponents $p,q$ of
\eqref{subc} means $-2+(n+1)/p+2/q<0$. This implies that the effect of $\tilde u$ becomes smaller as one looks at a finer scale 
$\lambda\approx 0$.
For this reason, it is fitting to say that the problem \eqref{prob} with \eqref{subc} is subcritical.  In contrast, the case of $n+1=p=q=2$ 
is a critical case which just 
fails to 
satisfy \eqref{subc}. 
The solution in the present
paper has the feature that the effect of curvature and that of forcing are analytically comparable in strength since the parabolic 
change of variables, which is invariant under the mean curvature flow, is also invariant for the 
norms of $u$ considered here. There is no known general existence theorem for the ``supercritical'' case, where 
at least one of the
inequalities in \eqref{subc} is 
$<$. We expect that there is no 
solution of the similar type in the sense
that the forcing field is too turbulent for the evolving surface to
even retain the finiteness of the
surface measure in general.  

We next discuss the technical difference between the critical and subcritical cases of the problem.
One important point for both cases is that, unlike the mean curvature flow ($u=0$) which simply reduces the surface area $\mathcal H^n(\Gamma(t))$ in time, $\mathcal H^n(\Gamma(t))$ may blow up immediately with irregular $u$.
The strategies of having a good bound for $\mathcal H^n(\Gamma(t))$
can be explained as follows.
If $u$
and $\Gamma(t)$ are smooth and 
\eqref{prob} is satisfied, then by the
first variation formula of the surface measure,
we have
\begin{equation}
    \frac{d}{dt}\mathcal H^n(\Gamma(t))=-\int_{\Gamma(t)}v\cdot h\,d\mathcal H^n=-\int_{\Gamma(t)}(|h|^2+h\cdot u^\perp)\,d\mathcal H^n.
\end{equation}
We may proceed to use $-h\cdot u^\perp\leq |h|^2/2+|u|^2/2$, so that we would have an apriori estimate for $\mathcal H^n(\Gamma(t))$ in term of $\mathcal H^n(\Gamma(0))$ if we have a 
control of the term
\begin{equation}
\int_0^T\int_{\Gamma(t)}|u|^2\,d\mathcal H^n dt.
\end{equation}
Note that $|u|^2$ is integrated with respect to the surface measure, so that we need to control the $L^2$-norm of the {\it trace} of Sobolev function $u$ on $\Gamma(t)$. For this purpose, we take advantage of the Meyers-Ziemer inequality \cite{meyers1977integral} (see Theorem \ref{MZ}), namely, if the 
upper density ratio 
\begin{equation}\label{updensity}
    \mathfrak{D}(t)=\sup_{B_r(x)\subset\mathbb R^{n+1}}\frac{\mathcal H^n(\Gamma(t))}{r^n}
\end{equation}
is bounded, then the Meyers-Ziemer
inequality gives
\begin{equation}
\int_{\Gamma(t)}|u|^2\,d\mathcal H^n\leq c(n)\mathfrak{D}(t)\int_{\mathbb R^{n+1}}|\nabla |u|^2|\,dx,
\end{equation}
and if the right-hand side is 
appropriately bounded, we would be able to obtain the apriori bound for
$\mathcal H^n(\Gamma(t))$. For the subcritical case of \eqref{subc} (as in \cite{takasao2016existence}), 
it turned out that the local $L^\infty$ bound of $\mathfrak{D}(t)$ in time
can be controlled using an analogue of Huisken's monotonicity formula \cite{Huisken_mono}. 
For the present critical case,
a similar strategy does not work since the contribution coming from $u$ cannot
be controlled in the computation of Huisken's monotonicity formula. Instead, we utilize a monotonicity formula
for one-dimensional general varifold
which gives
\begin{equation} \label{Dtest}
    \mathfrak{D}(t)\leq \int_{\Gamma(t)}|h|\,d\mathcal H^1
\end{equation}
in essence (see Proposition \ref{length} for the precise statement) and which allows us to
close the above estimate. In this sense, the strategy is particular to one-dimension and is different from the subcritical case. In the 
end, we only obtain
\begin{equation}
    \int_0^T \mathfrak{D}(t)^2\,dt\leq 
    \sup_{t\in[0,T]}\mathcal H^1(\Gamma(t))\int_0^T\int_{\Gamma(t)}|h|^2\,d\mathcal H^1
\end{equation}
and not necessarily the local $L^\infty$ bound
on $\mathfrak{D}(t)$ as in the subcritical case.

We also utilize 
global-in-time existence results \cite{kim2017mean,Kim-Tone2,ST-canonical}
for the mean curvature flow, which can be modified for the flow with a smooth forcing term, and which prove the global-in-time existence of approximate flows. 
We show a priori estimates 
on the length of the curve, $L^2$-norm of the curvature and
$L^2$-norm of the vector field $u$ (as a trace)
depending only on $\mathcal H^1(\Gamma_0)$ and $\Cr{c1}(u,T)$. It is 
proved that the limit of the
approximate flow satisfies a set of
desired properties
similar to the results in \cite{takasao2016existence}.
As in \cite{kim2017mean,Kim-Tone2,ST-canonical}, the 
setting of the problem is for
that of general multi-phase and the 
flow typically keeps triple junctions as it evolves. 

The organization of the paper is as follows. In the next Section 2, notation and main results are stated. In Section 3, key a priori estimates are presented.
Section 4 states the existence
results of the flow in the case
of smooth forcing term and for general dimensions, and whose proof is relegated to the Appendix. Section 5 
analyzes the limiting process 
and proves the 
main theorem, and in Section 6, we mention a few additional comments. 

\section{Notation and main theorem}
\subsection{Notation}
Though the main result of the present paper is for $n=1$, here we write the notation for general dimensions because of Section \ref{erff}. We define  $G_n(\R^{n+1}):=\R^{n+1}\times \mathbf{G}(n+1, n)$
   where $\mathbf{G}(n+1, n)$ is   the space of $n$-dimensional subspaces of $\R^{n+1}$. The
   subspace $S$ is often identified with the $(n+1)\times (n+1)$ matrix representing the 
   orthogonal projection $\R^{n+1}\rightarrow S$. 
   A general $n$-varifold is a Radon measure on $G_n(\R^{n+1})$ (see \cite{allard1972first,simon1983lectures} for more comprehensive presentation). The set of all general $n$-varifolds is denoted by $\mathbf{V}_n(\R^{n+1})$. For $V \in \mathbf{V}_n(\R^{n+1})$, let $\|V\|$ be the weight measure of $V$, namely,
$$
\|V\|(\phi):=\int_{G_n(\R^{n+1})} \phi(x)\, d V(x, S)
$$
for $\phi \in C_c(\R^{n+1})$. 
A set $M\subset\mathbb R^{n+1}$ is 
said to be {\it countably $n$-rectifiable} if there exists 
a countable number of Lipschitz functions $F_j:\mathbb R^n\rightarrow\mathbb R^{n+1}$ ($j\in\mathbb N$) such that $\mathcal H^n(M\setminus\cup_{j=1}^\infty F_j(\mathbb R^n))=0$ (see \cite{simon1983lectures} for the other properties). $V \in \mathbf{V}_n(\R^{n+1})$ is called  rectifiable if there exist an $\mathcal{H}^n$-measurable countably $n$-rectifiable set $M \subset \R^{n+1}$ and a locally $\mathcal{H}^n$-integrable function $\theta$ defined on $M$ such that
\begin{equation}\label{defrecti1varifold}
    V(\phi)=\int_M \phi\left(x, \operatorname{Tan}_x M\right) \theta(x)\, d \mathcal{H}^n(x)
\end{equation}
for $\phi \in C_c\left(G_n(\R^{n+1})\right)$. Here $\operatorname{Tan}_x M$ is the approximate tangent space of $M$ at $x$ which exists $\mathcal{H}^n$ a.e.~on $M$. Any rectifiable $n$-varifold is uniquely determined by its weight measure $\|V\|=\theta\, \mathcal{H}^{n} \lfloor_M$ through the formula \eqref{defrecti1varifold}. For this reason, we naturally say that  a Radon measure $\mu$ is rectifiable when one can associate a rectifiable varifold $V$ such that $\|V\|=\mu$. The approximate tangent
space $\operatorname{Tan}_x M$ is denoted in this case by 
$\operatorname{Tan}_x\mu$ or
$\operatorname{Tan}_x \|V\|$ without fear
of confusion. We let ${\bf var}(M,\theta)$ denote the varifold defined as in \eqref{defrecti1varifold}.
If $\theta \in \mathbb{N}$ for $\mathcal{H}^n$ a.e. on $M$, we say $V$ is integral. The set of all integral $n$-varifolds is denoted by $\mathbf{IV}_n(\R^{n+1})$. If $\theta=1$ for $\mathcal{H}^n$ a.e. on $M$, we say that  $V$ is a unit density $n$-varifold.

For $V \in \mathbf{V}_n(\R^{n+1})$ let $\delta V$ be the first variation of $V$, namely,
\begin{equation}\label{tonegawakimequ2.2}
    \delta V(g):=\int_{G_n(\R^{n+1})} \nabla g(x) \cdot S\, d V(x, S)\quad \text{ for }g \in C_c^1 (\mathbb{R}^{n+1} ; \mathbb{R}^{n+1} ).
\end{equation}
Here $\nabla g\cdot S:=\sum_{i,j=1}^{n+1} g^i_{x_j} S_{ij}$ which is the divergence of $g$ on $S$. If 
\begin{equation}
    \sup_{g\in C_c^1(\R^{n+1};\R^{n+1}), \,|g|\leq 1}\delta V(g)<\infty,
\end{equation}
we say that the first variation
is bounded, and in this 
case, 
 let $\|\delta V\|$ be the total variation measure. If $\|\delta V\|$ is absolutely continuous with respect to $\|V\|$, by the Radon-Nikodym theorem, we have a $\|V\|$-measurable vector field $h(x,V)$ such that 
\begin{equation}\label{tonegawakimequ2.2a}
    \delta V(g)=-\int_{\mathbb{R}^{n+1}} g(x) \cdot h(x,V) \,d\|V\|(x).
\end{equation}

The vector field $h(x,V)$ (often written as $h(V)$) is called the generalized mean curvature of $V$
(or curvature in the case of $n=1$).  For any $V \in \mathbf{IV}_n(\R^{n+1})$ with bounded first variation, Brakke's perpendicularity theorem for  generalized mean curvature \cite[Chapter 5]{brakke1978motion} says that 
\begin{equation}
    \label{tonegawakimequ2.3}
    S^{\perp}(h(x,V))=h(x,V)\quad \text{ for  $V$ a.e. }(x, S) \in G_n(\R^{n+1}),
\end{equation}
i.e.~$h(x,V)$ is perpendicular to the
approximate tangent space for $\|V\|$ a.e.~$x\in\R^{n+1}$.
For $V\in {\bf IV}_n(\R^{n+1})$ and 
$\|V\|$-integrable function $u$,
we use the notation
\begin{equation}
u^\perp(x):=(\operatorname{Tan}_x \|V\|)^\perp(u(x)),
\end{equation}
which is well-defined for $\|V\|$ a.e.~$x\in\R^{n+1}$. 

For a set of finite perimeter $E\subset\R^{n+1}$, let $\partial^* E$ be the reduced boundary of $E$. Let 
$\nabla \chi_E$ be
the distributional derivative of the 
characteristic function of $E$ and
let $\|\nabla \chi_{E}\|$ be the
total variation of $\nabla\chi_E$.
By the well-known structure theorem
for sets of finite perimeter (see \cite{evans2015measure}), $\mathcal H^n\lfloor_{\partial^* E}=\|\nabla\chi_E\|$. We also
write $|E|$ for the Lebesgue
measure of $E$. 

\subsection{Measure-theoretic formulation of the normal velocity}\label{mtfnv}
In this subsection we review what
it means for a one-parameter family of varifolds $\{V_t\}_{t\geq 0}$ to be a solution of
\eqref{prob}. To motivate the definition, let us assume first 
that we have a smooth family of 
$n$-dimensional surfaces without boundary $\Gamma(t)\subset\R^{n+1}$ moving by the 
normal velocity $v=v(x,t)$. For any
test function $\phi\in C^1_c(\mathbb R^{n+1}\times[0,\infty))$, we can prove that
\begin{equation}
    \frac{d}{dt}\int_{\Gamma(t)}\phi\,d\mathcal H^n(x)=\int_{\Gamma(t)} (\nabla\phi-\phi\,h)\cdot v+\frac{\partial\phi}{\partial t}\,d\mathcal H^n
\end{equation}
is satisfied. Conversely, if $\tilde v=\tilde v(x,t)$ is a smooth normal vector field defined on $\Gamma(t)$ which satisfies
\begin{equation}\label{smineqbra}
  \frac{d}{dt}\int_{\Gamma(t)}\phi\,d\mathcal H^n(x)\leq \int_{\Gamma(t)} (\nabla\phi-\phi\,h)\cdot \tilde v+\frac{\partial\phi}{\partial t}\,d\mathcal H^n   
\end{equation}
for all {\it non-negative} 
$\phi\in C_c^1(\mathbb R^{n+1}\times[0,\infty))$, then 
one can prove that $\tilde v$ has to
be the normal velocity vector, i.e., $\tilde v=v$ (see \cite[Section 2.1]{tonegawa2019brakke} for the
proof). Thus \eqref{smineqbra}
can be used to characterize the normal velocity vector, and
by considering the integrated version, we have the following 
\begin{proposition}
    Suppose $\{\Gamma(t)\}_{t\in[0,T]}$ is a smooth family of $n$-dimensional surfaces without  boundary. Then a normal vector
    $\tilde v$ defined on $\Gamma(t)$ coincides with the normal velocity vector $v$ if and only if
    the following holds. For any
    $0\leq t_1<t_1\leq T$ and for 
    any non-negative $\phi\in C_c^1(\mathbb R^{n+1}\times[0,T])$, 
    \begin{equation}\label{heuineq}
        \int_{\Gamma(t_2)}\phi(x,t_2)\,d\mathcal H^n(x)-\int_{\Gamma(t_1)}\phi(x,t_1)\,d\mathcal H^n(x)
        \leq\int_{t_1}^{t_2}\int_{\Gamma(t)}(\nabla\phi-\phi\,h)\cdot\tilde v+\frac{\partial\phi}{\partial t}\,d\mathcal H^n dt
    \end{equation}
    is satisfied. 
\end{proposition}
Motivated by this, we may define the normal velocity of a 
family of varifolds $\{V_t\}_{t\in[0,T]}$. First, as a 
generalization of the $n$-dimensional surface, we may require $V_t\in {\bf IV}_n(\mathbb R^{n+1})$ for 
a.e.~$t\in[0,T]$. Since we want to 
consider $\tilde v=h+u^\perp$, we require 
that $V_t$ has the generalized 
mean curvature vector $h(x,V_t)$ 
for a.e.~$t$, which is also $L^2$-integrable, i.e., $\int_0^T\int_{\mathbb R^{n+1}}|h(x,V_t)|^2\,d\|V_t\|dt<\infty$. Under these additional conditions and replacing $\tilde v$
in \eqref{heuineq} by $h+u^\perp$, 
$\mathcal H^n\lfloor_{\Gamma(t)}$ by $\|V_t\|$, respectively, we say that $\{V_t\}_{t\in[0,T]}$ is a solution of \eqref{prob} if 
\eqref{heuineq} holds true for
any $0\leq t_1<t_2\leq T$ and for any
non-negative $\phi\in C_c^1(\mathbb R^{n+1}\times[0,T])$.
See (V3) (which shows, in particular, that $V_t\in {\bf IV}_1(\mathbb R^2)$ for a.e.~$t$), (V4), (V5) and (V7) in the next subsection. 

\subsection{Main results}
We first state the assumptions for the
main theorem.
\begin{assumption}\label{goodas}
We consider the following (A1)-(A4):
    \begin{itemize}
    \item[(A1)] $N$ is an integer with $N\geq 2$.
    \item[(A2)] $E_{0,1},\ldots,E_{0,N}$ are non-empty open sets in $\mathbb R^{2}$ which
    are mutually disjoint.
    \item[(A3)] $\Gamma_0:=\mathbb R^{2}\setminus
    \cup_{i=1}^N E_{0,i}$ is countably 
    $1$-rectifiable and $\mathcal H^1(\Gamma_0)<\infty$. 
    \item[(A4)] $u$
    is a given vector field 
    as in \eqref{NS} and 
    $\Cr{c1}(u,T)$ is as in \eqref{NS1}.
\end{itemize}  
\end{assumption}
Here, the simplest case is $N=2$, which
may be considered as a ``two-phase'' case. In general, $N$ 
corresponds to the number of ``phases'' $E_{0,1},\ldots, E_{0,N}$ and they partition the whole $\R^2$. Note that the connectedness of each of $E_{0,1},\ldots,E_{0,N}$ is not assumed. Since $\Gamma_0$ cannot have interior
due to $\mathcal H^1(\Gamma_0)<\infty$, 
we have
$\Gamma_0=\cup_{i=1}^N\partial E_{0,i}$ and the initial datum 
$\Gamma_0$ for the flow is given as the topological
boundary of the partition. 

We next state the main existence theorem as
Theorem \ref{mainth1}, and we give an itemized
explanation in Section \ref{domr}.

\begin{theorem}\label{mainth1}
Assume {\rm (A1)-(A4)}. Then there exist a family $\{V_t\}_{t\geq 0}\subset{\bf V}_1(\mathbb R^{2})$ and a family of
sets of finite perimeter $\{E_i(t)\}_{t\geq 0}$ 
in $\R^{2}$ for each $i=1,\ldots, N$ with the following:
\newline
Properties for $V_t$:
\begin{itemize}
\item[(V1)] $\|V_0\|=\mathcal H^1\lfloor_{\Gamma_0}$.
\item[(V2)] If $\mathcal H^1(\cup_{i=1}^N (\partial E_{0,i}\setminus \partial^* E_{0,i}))=0$,
then $\lim_{t\rightarrow 0+}\|V_t\|=\|V_0\|$.  
    \item[(V3)] For a.e.~$t\geq 0$, for $t$
    with $\|V_t\|(\mathbb R^2)>0$, there 
exist a finite number (denoted by $P(t)\in\mathbb N$) of connected embedded $W^{2,2}$ curves $\ell_1(t),\ldots,\ell_{P(t)}(t)$,
each of them being
either a
closed curve or a curve  with two endpoints,
such that 
\begin{equation}\label{curvar}
V_t=\sum_{i=1}^{P(t)}{\bf var}(\ell_i(t),1).
\end{equation}
Here, it is not ruled out that 
some $\ell_i(t)$ 
and $\ell_{j}(t)$ with $i\neq j$ 
may be identical. 
Moreover, there is a finite set of points where the endpoints of curves $\ell_1(t),\ldots,\ell_{P(t)}(t)$ meet at angles of 
either 0, 60 or 120 degrees.
If $N=2$, each curve is an embedded closed curve which may intersect other curves only tangentially. 
\begin{figure}  [htp]
\centering
\includegraphics[scale=1.0] {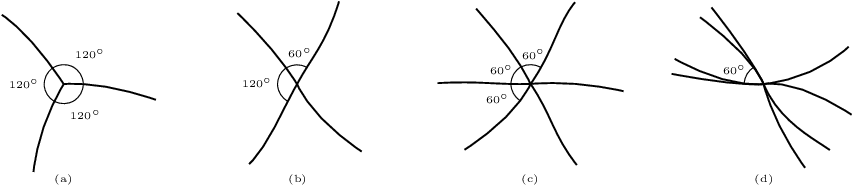}
\caption{}
\end{figure}
    \item[(V4)] For a.e.~$t\geq 0$, $\delta V_t$
    is bounded and absolutely continuous 
    with respect to $\|V_t\|$
    (thus $h(x,V_t)$ exists).
    \item[(V5)] There exists an absolute
    constant $\Cr{MZ}>0$ (cf. Theorem \ref{MZ}) such that for all $T>0$, 
    and writing $\Cr{c1}(u,T)$ (cf. \eqref{NS1}) as $\Cr{c1}$, 
    \begin{equation}\label{h2}
    \int_0^T\int_{\mathbb R^{2}}|h(x,V_t)|^2\,d\|V_t\|dt
    \leq 4\mathcal H^1(\Gamma_0)(1+\Cr{MZ}
    \Cr{c1}\exp(\Cr{MZ} \Cr{c1})),
    \end{equation}
      \begin{equation}\label{gaa1main}
\int_{0}^{T}\int_{\R^2}|u|^2\,d\|V_t\|dt \leq 4\sqrt{\Cr{c1}\Cr{MZ}}\mathcal H^1(\Gamma_0)\sqrt{1+\Cr{MZ}\Cr{c1}\exp(\Cr{MZ}\Cr{c1})}\exp(\Cr{MZ}\Cr{c1}/2).
\end{equation}
    \item[(V6)] For all $T>0$, 
\begin{equation}
\sup_{t\in[0,T]}\|V_t\|(\mathbb R^{2})\leq \mathcal H^1(\Gamma_0)\exp\left(\Cr{MZ}\,\Cr{c1}(u,T)\right).
    \end{equation}
    \item[(V7)] 
    For all $0\leq t_1<t_2<\infty$ and $\phi\in C^1_c(\mathbb R^{2}\times[0,\infty))$ with
    $\phi\geq 0$, it holds
    \begin{equation}\label{bra}
    \begin{split}
    &\|V_{t_2}\|(\phi(\cdot,t_2))-\|V_{t_1}\|(\phi(\cdot,t_1)) \leq \int_{t_1}^{t_2}\int_{\mathbb R^{2}}
   \frac{\partial\phi}{\partial t}(x,t)
    \,d\|V_t\|(x)dt\\
    & +\int_{t_1}^{t_2}\int_{\mathbb R^{2}}
    \big\{\nabla\phi(x,t)-\phi(x,t)h(x,V_t)\big\}
    \cdot \big\{h(x,V_t)+u^\perp(x,t)\big\}
    \,d\|V_t\|(x)dt.
    \end{split}
    \end{equation}
\end{itemize}
Properties of $E_i(t)$:
\begin{itemize}
\item[(E1)] $E_i(0)=E_{0,i}$ for $i=1,\ldots,N$.
\item[(E2)] $E_1(t),\ldots,E_N(t)$ are mutually disjoint sets of
finite perimeter for all $t\geq 0$.
\item[(E3)] For all $t\geq 0$, $\|V_t\|\geq \|
\nabla
\chi_{E_i(t)}\|$ for each $i=1,\ldots, N$
 and 
 $2\|V_t\|\geq \sum_{i=1}^N\|\nabla\chi_{E_i(t)}\|$.
 \item[(E4)] $S(i):=\{(x,t): x\in E_i(t), t\geq 0\}$ is a set of 
 finite perimeter in $\mathbb R^{2}\times[0,T)$ 
for all $T>0$ and $i=1,\ldots,N$.
\item[(E5)] There exists a constant $\Cl[con]{vol2}=\Cr{vol2}(\Cr{c1}(u,T),\mathcal H^1(\Gamma_0))$ such that, for each $i=1,\ldots,N$
and $0\leq t_1<t_2\leq T$, 
\begin{equation}
    |E_i(t_2)\triangle E_i(t_1)|
    \leq \Cr{vol2}(t_2-t_1)^{\frac12}.
\end{equation}
\item[(E6)] There exists a
Radon measure $\beta$ on $[0,\infty)$ with $\beta([0,T])\leq \Cl[con]{vol}=\Cr{vol}(\Cr{c1}(u,T),\mathcal H^1(\Gamma_0))$ 
such that, 
for any $0\leq t_1<t_2\leq T$,  
$\phi\in C_c(\mathbb R^{2}\times[0,\infty))$
and $i=1,\ldots,N$, we have
\begin{equation}\label{gvel2}
\begin{split}
\Big|\int_{E_i(t_2)}&\phi(x,t_2)\,dx-
\int_{E_i(t_1)}\phi(x,t_1)\,dx-\int_{t_1}^{t_2}\int_{E_i(t)} 
\frac{\partial\phi}{\partial t}\,dxdt
\Big| \\ &
    \leq (\beta([t_1,t_2]))^{1/2}\left(\int_{t_1}^{t_2}\int_{\R^2}\phi^2\,d\|V_t\|dt\right)^{1/2}.
    \end{split}
\end{equation}
\item[(E7)] Let $\theta_t(x)$ be the density of $\|V_t\|$.
Then for a.e.~$t\geq 0$ and $\mathcal H^1$ a.e.~$x\in \cup_{i=1}^{P(t)}\ell_i(t)$
(cf.~\eqref{curvar})
\begin{itemize}
\item[(1)] if $N\geq 3$, then, $\theta_t(x)=1$ implies
$x\in \cup_{i=1}^N\partial^* E_i(t)$.
\item[(2)] If $N=2$, 
\begin{equation}
\theta_t(x)=\left\{\begin{array}{ll}\mbox{odd integer for } x\in\partial^* E_1(t)(=\partial^* E_2(t)), \\
\mbox{even integer for }
x\in \cup_{i=1}^{P(t)}\ell_i(t)\setminus \partial^* E_1(t).
\end{array}
\right.
\end{equation}
\end{itemize}
\item[(E8)]
If $\{V_t\}_{0\leq t\leq T}$ is
of unit density for a.e.~$t\in[0,T]$,
then $V_t={\bf var}(\cup_{i=1}^N
\partial^* E_i(t),1)$ for a.e.~$t\in[0,T]$.
\end{itemize}
\end{theorem}
\subsection{Discussion on the main results}\label{domr}
We give comments on the stated results. 
\begin{itemize}
\item[(V1)] This specifies that the initial condition is satisfied.
\item[(V2)] In general, $\partial^* E_{0,i}\subset\partial E_{0,i}$, that is,  the measure-theoretic boundary is contained in the topological boundary, but they are not equal when there is a
piece of topological boundary ``within the set'', to put it 
heuristically. Such a set typically vanishes 
instantaneously under our construction, so that we cannot expect the stated continuity in general. On the other hand, if the difference has null measure as assumed in (V2), then $\|V_t\|$ is continuous
at $t=0$. 
\item[(V3)] This states that, for a.e.~$t\geq 0$ and for $t$ with $\|V_t\|(\mathbb R^2)>0$, the flow 
looks like a network of $W^{2,2}$ curves with triple 
junctions with 120-degree angles or the 
degenerate triple junctions (such as ``infinitesimal hexagons'' as in Figure 1(c)).
This regularity is a special feature of the 
construction of \cite{kim2017mean,Kim-Tone2}.
For the case of $N=2$ (``two-phase''), there is no
junction and each curve is an embedded and closed loop, 
and they may only intersect tangentially. 
We know the local description of the support of $\|V_t\|$, thus the number of curves is
finite, which we denote  as $P(t)$. On the other hand, we do not know
any further property of $P(t)$.   
\item[(V4)] This property in particular implies that the sum of the unit conormals of the curves
at the junction (counting multiplicities) must be $0$.
\item[(V5,6)]
These give control of the relevant
integrals in terms of the initial length and 
norm of $u$ in an explicit manner. 
\item[(V7)]
The inequality \eqref{bra} is a 
characterization of the normal velocity being equal 
to $h+u^\perp$ as described in Section \ref{mtfnv} in a generalized sense of Brakke.
Unlike the subcritical case of \cite{takasao2016existence} where the regularity
theorem of \cite{kasai2014general} is applicable,
no regularity theorem of a similar kind is known
for this flow due to the criticality. 
\end{itemize}
We next discuss the results on $E_i(t)$. 
\begin{itemize}
\item[(E1)] Each $E_i(t)$ starts with the given
initial open set $E_{0,i}$. As already stated, $E_{0,i}$ may not be connected in general.
\item[(E2)] 
At each time $t$, $E_1(t),\ldots,E_N(t)$ give
a partition of $\R^2$, and note that some of them
may be empty sets. Unlike the subcritical case, where $E_i(t)$ is open (see (E2')), we  can only conclude that $E_i(t)$ is a set of finite perimeter.
\item[(E3)] The perimeter measure of $E_i(t)$
bounds $\|V_t\|$ from below and this property is 
useful to guarantee the non-triviality of $\|V_t\|$
in the sense that, as long as there are at least 
two elements of $E_1(t),\ldots,E_N(t)$
with positive Lebesgue measure, we have 
$\|V_t\|\neq 0$. 
\item[(E4)] Each $E_i(t)$ may be obtained as 
the time-slice of $S(i)\subset\mathbb R^2\times[0,\infty)$, which is itself a set of finite perimeter in $\mathbb R^2\times[0,T]$ for all $T$.
\item[(E5)] The Lebesgue
measure of $E_i(t)$ is H\"{o}lder continuous in
$t$.  In particular, it cannot suddenly vanish discontinuously. 
\item[(E6)] More detailed information is contained in \eqref{gvel2}
on the continuity of each $E_i(t)$.
\item[(E7)]
This is an intuitively expected result in that, if 
the density is one, then, there should be only 
one curve with unit density in the neighborhood, so that there should
be exactly two $E_i(t)$'s whose boundaries 
coincide with the curve. More precise information
is available if $N=2$. If 
the multiplicity is $\theta\geq 2$, then, there are 
$\theta$ curves counting multiplicities. Crossing each
curve transversely, one enters from $E_1(t)$ to $E_2(t)$ or vice-versa. Thus, depending on the
parity, the point is either in the reduced boundary or
inside the domain. The idea of proof is somewhat 
similar in spirit. 
\item[(E8)]
If we have a unit 
density flow, then the measure $\|V_t\|$ 
is precisely the one coming from the reduced
boundary of $E_i(t)$. In other
words, the higher multiplicity occurs with positive space-time measure only when
the discrepancy $\|V_t\|>{\bf var}(\cup_{i=1}^N\partial^* E_i(t),1)$ happens with positive
measure of time. 
\end{itemize}
We give more comments in Section \ref{erff}
comparing the subcritical case and the present critical case in particular. 

\section{Key a priori estimates}
In this section, we establish the length estimate 
\eqref{NS2} for an evolving 
curve with the velocity 
given by \eqref{NS} in the sense
expressed in \eqref{bra}. We also obtain the $L^2$ control of $u$. 
The first proposition is for a fixed time, so we drop the reference to $t$ in $V_t$. The result is certainly not  new and we include the
proof from  \cite{MR4158524} for the
reader's convenience.
\begin{proposition}\label{length}
Let $V \in  \mathbf{V}_1(\R^2)$ with  $\|V\|(\R^2)<\infty$
and $\|\delta V\|(\R^2)<\infty$. Then the following quantity   \begin{equation} 
    \mathfrak{D}(\|V\|):=    \sup_{B_r(x)\subset\mathbb R^2} \frac{\|V\|(B_r(x))}{r}   \end{equation}
    satisfies the estimate
    \begin{equation}\label{lenest}
    \mathfrak{D}(\|V\|)      \leq \|\delta V\|(\R^2).
    \end{equation}
    If in addition that $\|\delta V\|$ is absolutely 
    continuous with respect to $\|V\|$, then we have
    \begin{equation}\label{lenest2}
      \mathfrak{D}(\|V\|)\leq (\|V\|(\R^2))^{\frac12}\big(\int_{\R^2}|h(V)|^2\,d\|V\|\big)^{\frac12}.  
    \end{equation}
\end{proposition}
\begin{proof}
For  $0<\sigma<\rho$ and $x_0\in\R^2$, we consider  the vector field $$X(x)=\left(\frac{1}{\left|x-x_0\right|_\sigma}-\frac{1}{\rho}\right)_{+} (x-x_0),$$  where $(\cdot)_{+}$ denotes the nonnegative part and $|\cdot|_\sigma=$ $\max \{|\cdot|, \sigma\}$.  
For $S\in {\bf G}(2,1)$, we can 
compute as
$$
S\cdot \nabla X(x)= \begin{cases}\frac{1}{\sigma}-\frac{1}{\rho} & \text { on }  B_\sigma(x_0), \\ \frac{\left|S^\perp\left(x-x_0\right)\right|^2}{\left|x-x_0\right|^3}-\frac{1}{\rho} & \text { on } B_\rho(x_0)\backslash \bar{B}_\sigma(x_0) ,\end{cases}
$$
where $S^\perp$ represents the orthogonal projection to the
orthogonal complement of $S$.
  Integrating   $S\cdot\nabla X$   with respect to $V$,
  
  $$
  \delta V(X)=\frac{\|V\|(B_\sigma(x_0))}{\sigma}-\frac{\|V\|(B_\rho(x_0))}{\rho}
  +\int_{G_1(B_\rho (x_0)\setminus
  B_\sigma(x_0))} \frac{|S^\perp(x-x_0)|^2}{|x-x_0|^3}
  \,dV(x,S)$$
  while using $|X|\leq 1$, we
  have
  $$
  \frac{\|V\|(B_\sigma(x_0))}{\sigma}-\frac{\|V\|(B_\rho(x_0))}{\rho}\leq \|\delta V\|(\R^2).
  $$
   By letting $\rho\rightarrow \infty$, we obtain \eqref{lenest}
   and using \eqref{tonegawakimequ2.2a}, 
  we have \eqref{lenest2}.
\end{proof}
We next recall the Meyers-Ziemer inequality \cite{meyers1977integral} that bridges the 
integral with respect to a Radon measure and the Lebesgue measure of compactly supported functions. 
\begin{theorem}\label{MZ} 
There exists a constant $\Cr{MZ}>0$ depending 
only on $n$ with the following property. 
 Suppose that $\mu$ is a Radon 
measure on $\mathbb R^{n+1}$ with 
\begin{equation}\label{le1}
    \mathfrak{D}(\mu):=\sup_{B_r(x)\subset\mathbb R^{n+1}}
    \frac{\mu(B_r(x))}{r^{n}}<\infty.
\end{equation}
Then for any $\phi\in C^1_c(\mathbb R^{n+1})$, we have
\begin{equation}\label{le2}
    \int_{\mathbb R^{n+1}}\phi(x)\,d\mu(x)\leq \sqrt{\Cr{MZ}}\,\mathfrak{D}(\mu)\int_{\mathbb R^{n+1}}|\nabla \phi(x)|\,dx.
\end{equation}
\end{theorem} 
We use Theorem \ref{MZ} 
and Proposition \ref{length}
to obtain the following a priori estimate for flows satisfying 
\eqref{prob} with a regular $u$. 
\begin{proposition}\label{lengthcontrol}
    Consider a family of varifolds  
  $\left\{V_t\right\}_{t \geq 0} \subset \mathbf{V}_{1}(\R^2)$ 
  with the following properties.
  \begin{itemize}
  \item[(a)] For a.e.~$t\geq 0$, $V_t\in {\bf IV}_1(\R^2)$, 
  $\delta V_t$ is bounded and absolutely continuous with respect to $\|V_t\|$, and $\int_0^T\int_{\R^2}|h(x,V_t)|^2\,d\|V_t\|dt<\infty$ for all $T>0$.
      \item[(b)] For all $T>0$, $\sup_{t\in[0,T]}\|V_t\|(\R^2)<\infty$.
\item[(c)] For some $u\in C_c^1(\R^2\times[0,\infty);\R^2)$,
      the property (V7) in Theorem \ref{mainth1} is satisfied. 
  \end{itemize}
  Set $\Cr{c1}=\Cr{c1}(u,T)$ as in \eqref{NS1}. Then we have
    \begin{equation}\label{goo3}
        \sup_{t\in[0,T]}\|V_t\|(\R^2)\leq 
        \|V_0\|(\R^2)\exp(\Cr{MZ}\,
        \Cr{c1}),
    \end{equation}
    \begin{equation}\label{ga1}
\int_0^T\int_{\R^2}|h(V_t)|^2\,d\|V_t\|dt\leq 4\|V_0\|(\R^2)(1+\Cr{MZ}\Cr{c1}\exp(\Cr{MZ}\,\Cr{c1})),
    \end{equation}
    \begin{equation}\label{gaa1}
\int_{0}^{T}\int_{\R^2}|u|^2\,d\|V_t\|dt \leq 4\sqrt{\Cr{c1}\Cr{MZ}}\|V_0\|(\R^2)\sqrt{1+\Cr{MZ}\Cr{c1}\exp(\Cr{MZ}\Cr{c1})}\exp(\Cr{MZ}\Cr{c1}/2).
\end{equation}
\end{proposition}
\begin{proof}
We consider a sequence $\{\phi_\ell\}_{\ell\in\mathbb N}\subset C_c^1\left(\R^2; [0,1]\right)$ which equals  $1$ in $B_\ell$ and vanishes outside $B_{2\ell}$ with $|\nabla\phi_\ell|\leq 1/\ell$.  Using \eqref{bra} and the dominated convergence theorem, we may
take the limit $\ell\to \infty$ in \eqref{bra} with $\phi=\phi_\ell$ and obtain
\begin{equation} \begin{split}
    \left.\|V_t\|(\R^2)\right|_{t=t_1} ^{t_2}+&\int_{t_1}^{t_2} \int_{\R^2}|h(V_t)|^2\, d\|V_t\|dt\leq \int_{t_1}^{t_2}  \int_{\R^2} -h(V_t) \cdot u^\bot  \,d\|V_t\|dt\\
    &\leq\frac12\int_{t_1}^{t_2}\int_{\R^2}
|h(V_t)|^2\,d\|V_t\|dt+\frac12\int_{t_1}^{t_2}\int_{\R^2}
    |u|^2\,d\|V_t\|dt.
\end{split}\end{equation}
To estimate the last integral, we argue for a.e. $t\in [t_1,t_2]$ that 
\begin{equation}\label{u2Vt}
\begin{split}
    \int_{\R^2} & |u |^2\,d\|V_t \| \overset{\eqref{le2}}\leq \sqrt{\Cr{MZ}}\,\mathfrak{D}( \|V_t \|)\int_{\R^2}|\nabla |u |^2|\,dx \\&
    \overset{\eqref{lenest2}}\leq
    2\sqrt{\Cr{MZ}}(\|V_t\|(\R^2))^{\frac12}
    \big(\int_{\R^2}|h(V_t)|^2\,d\|V_t\|
    \big)^{\frac12}\big(\int_{\R^2}|\nabla u|^2\,dx\big)^{\frac12}\big(\int_{\R^2}|u|^2\,dx\big)^{\frac12}\\
    &\leq 
    \frac 12 \int_{\R^2} |h(V_t)|^2\, d\|V_t\|  +2\Cr{MZ}
\|V_t\|(\R^2)\int_{\mathbb R^2}|\nabla u|^2\,dx
\int_{\mathbb R^2}|u|^2\,dx.    \end{split}
    \end{equation}
  The above  two inequalities together imply 
  \begin{equation} \label{goo4}
  \begin{split}
    &\left.\|V_t\|(\R^2)\right|_{t=t_1} ^{t_2}+\frac 14\int_{t_1}^{t_2}\int_{\R^2} |h(V_t)|^2\, d\|V_t\| dt\\
    &\leq \Cr{MZ}\sup_{t\in[t_1,t_2]}\|u(\cdot,t)\|_{L^2}^2\int_{t_1}^{t_2} \big( \|V_t\|(\R^2)
\int_{\mathbb R^2}|\nabla u|^2\,dx\big)dt.
\end{split}
\end{equation}
This combined with Gronwall's inequality leads to \eqref{goo3}. Using \eqref{goo3} in 
\eqref{goo4}, we obtain \eqref{ga1},
and using \eqref{goo3} and \eqref{ga1} in \eqref{u2Vt}, one obtains \eqref{gaa1}. 
\end{proof}
\section{Existence result for flow with regular forcing}\label{erff}
Given a vector field with a compact support $u\in C_c^1(\mathbb R^{n+1}\times[0,\infty);\mathbb R^{n+1})$ and a suitable initial datum, by modifying the proof of \cite{kim2017mean,Kim-Tone2,ST-canonical}, we may establish the existence
of a flow of multi-phase boundary moving with  $v=h+u^\perp$. Since the modifications are 
somewhat of technical nature, we relegate the
proof to Appendix \ref{append}.
We state the 
result for general dimensions since
no change is needed. The stated results are not
exhaustive and the reader is referred to \cite{ST-canonical} for additional properties, such as \cite[Theorem 2.10,2.12,2.13]{ST-canonical}, which
hold equally for the current situation.
\begin{assumption}\label{asp}
Consider the following ${\rm (A1)}$--${\rm (A3)}$, ${\rm (A4')}$, $n\in \mathbb N$:     \begin{itemize}
     \item[(A1)] $N$ is an integer with $N\geq 2$.
     \item[(A2)] $E_{0,1},\ldots,E_{0,N}$ are non-empty open sets in $\mathbb R^{n+1}$ which
     are mutually disjoint.
     \item[(A3)] $\Gamma_0:=\mathbb R^{n+1}\setminus
     \cup_{i=1}^N E_{0,i}$ is countably 
     $n$-rectifiable and $\mathcal H^n(\Gamma_0)<\infty$. 
     \item[(A4')] $u$ is a vector field in $C^1_c(\mathbb R^{n+1}\times[0,\infty);\mathbb R^{n+1})$.
 \end{itemize}
 \end{assumption}
\begin{theorem}\label{exreg}
Assume ${\rm (A1)-(A3),(A4')}$. 
Then 
there exist a family $\{V_t\}_{t\geq 0}\subset{\bf V}_n(\mathbb R^{n+1})$ and a family of
open sets $\{E_i(t)\}_{t\geq 0}$ for each $i=1,\ldots, N$ with the following properties:
\newline
Properties for $V_t$:
\begin{itemize}
\item[(V1')] $\|V_0\|=\mathcal H^n\lfloor_{\Gamma_0}$.
\item[(V2')] If $\mathcal H^n(\cup_{i=1}^N (\partial E_{0,i}\setminus \partial^* E_{0,i}))=0$,
then $\lim_{t\rightarrow 0+}\|V_t\|=\|V_0\|$.
    \item[(V3')] For a.e.~$t\geq 0$, $V_t\in {\bf IV}_n(\mathbb R^{n+1})$,
    and if $n=1$, ${\rm (V3)}$ of Theorem \ref{mainth1} holds.
    \item[(V4')] For a.e.~$t\geq 0$, $\delta V_t$
    is bounded and absolutely continuous 
    with respect to $\|V_t\|$ (thus $h(x,V_t)$ exists).
    \item[(V5')] For all $T>0$,  
    $h(\cdot,V_t)$ satisfies 
    \begin{equation}
    \int_0^T\int_{\mathbb R^{n+1}}|h(x,V_t)|^2\,d\|V_t\|(x)dt<\infty.
    \end{equation}
    \item[(V6')] For all $T>0$, $\sup_{t\in[0,T]}\|V_t\|(\mathbb R^{n+1})<\infty$.
    \item[(V7')] 
    For all $0\leq t_1<t_2<\infty$ and $\phi\in C^1_c(\mathbb R^{n+1}\times[0,\infty))$ with $\phi\geq 0$, it holds
    \begin{equation}\label{bra-n}
    \begin{split}
    &\|V_{t_2}\|(\phi(\cdot,t_2))-\|V_{t_1}\|(\phi(\cdot,t_1)) \leq \int_{t_1}^{t_2}\int_{\mathbb R^{n+1}}
   \frac{\partial\phi}{\partial t}(x,t)
    \,d\|V_t\|(x)dt\\
    & +\int_{t_1}^{t_2}\int_{\mathbb R^{n+1}}
    \big\{\nabla\phi(x,t)-\phi(x,t)h(x,V_t)\big\}
    \cdot \big\{h(x,V_t)+u^\perp(x,t)\big\}
    \,d\|V_t\|(x)dt.
    \end{split}
    \end{equation} 
\end{itemize}
Properties of $E_i(t)$:
\begin{itemize}
\item[(E1')] $E_i(0)=E_{0,i}$ for $i=1,\ldots,N$.
\item[(E2')] $E_1(t),\ldots,E_N(t)$ are mutually disjoint open sets for all $t\geq 0$.
\item[(E3')] Writing $\Gamma(t):=\mathbb R^{n+1}
\setminus \cup_{i=1}^N E_i(t)$, we have $\mathcal H^n(\Gamma(t))<\infty$ and $\Gamma(t)=\cup_{i=1}^N\partial E_i(t)$ for all $t>0$.
\item[(E4')] Write $d\mu:=d\|V_t\|dt$ and $({\rm spt}\,\mu)_t:=\{x\in\mathbb R^{n+1}:(x,t)\in {\rm spt}\,\mu\}$. Then $\Gamma(t)=({\rm spt}\,\mu)_t$ for all $t>0$. 
\item[(E5')] For all $t\geq 0$, $\|V_t\|\geq \|
\nabla
\chi_{E_i(t)}\|$ for each $i=1,\ldots, N$
 and 
 $2\|V_t\|\geq \sum_{i=1}^N\|\nabla\chi_{E_i(t)}\|$.
 \item[(E6')] $S(i):=\{(x,t): x\in E_i(t), t\geq 0\}$ is open in $\mathbb R^{n+1}\times[0,\infty)$ 
for $i=1,\ldots,N$.
\item[(E7')] For any $0\leq t_1<t_2<\infty$,  
$\phi\in C^1_c(\mathbb R^{n+1}\times[0,\infty))$
and $i=1,\ldots,N$, we have
\begin{equation}\label{gvel}
\int_{E_i(t)}\phi(x,t)\,dx\Big|_{t=t_1}^{t_2}
    =\int_{t_1}^{t_2}\Big(\int_{\partial^* E_i(t)}(h+u)\cdot\nu_i\,\phi\,d\mathcal H^n
    +\int_{E_i(t)}\frac{\partial\phi}{\partial t}\,dx\Big)\,dt.
\end{equation}
Here $\nu_i$ is the outer unit-normal of the
reduced boundary $\partial^* E_i(t)$. 
\item[(E8')] Let $\theta_t(x)$ be the density of $\|V_t\|$.
Then for $\mathcal L^1$-a.e.\,$t\geq 0$ and $\mathcal H^n$-a.e.\,$x\in\Gamma(t)$,
\begin{itemize}
\item[(1)] if $N\geq 3$, then $\theta_t(x)=1$ implies
$x\in \cup_{i=1}^N\partial^* E_i(t)$.
\item[(2)] If $N=2$, 
\begin{equation}
\theta_t(x)=\left\{\begin{array}{ll}\mbox{odd integer for } x\in\partial^* E_1(t)(=\partial^* E_2(t)), \\
\mbox{even integer for }
x\in \Gamma(t)\setminus \partial^* E_1(t).
\end{array}
\right.
\end{equation}
\end{itemize}
\end{itemize}
\end{theorem}
We discuss the difference between the critical
and subcritical cases. Overall results are 
similar but the subcritical case gives 
better results due to the availability of 
Huisken's monotonicity formula. With it, one
can obtain an upper density ratio bound
\eqref{le1} uniform in time for any fixed 
time interval $[0,T]$. For the critical case, 
as one can see in \eqref{lenest2}, we only
have a time-integrated $L^2$ bound on $\mathfrak{D}(\|V_t\|)$
which is a substantially weaker control. Thus,
compared to (E2') where $E_i(t)$ is open, in (E2) we only have
 measurable $E_i(t)$ and a similar difference holds between (E6') and
(E4). This also causes a certain weaker conclusion
on (E6) compared to (E7'), where the result is 
reduced to the estimate \eqref{gvel2} instead of
the identity \eqref{gvel}. Technically speaking, 
we do not know if $d\|V_t\|dt$ is absolutely 
continuous with respect to $\mathcal H^2$ in $\R^2\times[0,\infty)$ for the critical case and
this gives a difficulty in establishing the identity
\eqref{gvel} in \cite{ST-canonical}.
\section{Proof of main theorem}
For the given vector field $u$ as in \eqref{NS}
and $m\in\mathbb N$, we first extend $u$ by zero for $t\notin [0,m]$ and we shall consider the mollified vector field  
\begin{align}\label{mollifyu}
u^{(m)}(x,t)=\int_{\R}\int_{\mathbb{R}^2}\zeta_{1/m}(x-y)\,\varrho_{1/m}(t-s)\,  u(y,s)\,dyds\,\eta (|x|/m)
\end{align}
where $\zeta_{\epsilon}$ and $\varrho_{\epsilon}$ are
respectively 2-d and 1-d standard 
mollifiers, and $\eta$ is in $C^{\infty}([0,\infty))$ with the property that $0\leq \eta\leq 1$, $\eta(s)=1$ for $0\leq s\leq 1$ and $\eta(s)=0$ for $s\geq 2$.
Then $u^{(m)}$ is a $C^\infty$ vector
field with compact support. By the standard argument (see   \cite[Theorem 2.29 and Lemma 3.16]{MR2424078}), one can prove 
for any $T>0$ that
\begin{align}
\label{ucon1}
    \lim_{m\to\infty}\int_0^T \|u^{(m)}(\cdot,t)-u(\cdot,t)\|_{W^{1,2}(\R^2)}^2\,dt=0
\end{align}
and
\begin{align}
    \label{uconin}
    \sup_{t\in[0,T]}\|u^{(m)}(\cdot,t)\|_{L^2(\mathbb R^2)}\leq
\esssup_{t\in[0,T+1]}\|u(\cdot,t)\|_{L^2(\mathbb R^2)}.
\end{align}
In the remaining part of this section, we let $\{V^{(m)}\}_{t\geq 0}$ and 
$\{E_i^{(m)}(t)\}_{t\geq 0}$ ($i=1,\ldots,N$) be the 
solution corresponding to $u=u^{(m)}$
in Theorem \ref{exreg}, and 
we prove that the desired solution
can be obtained as a suitable limit
of $V_t^{(m)}$ and $E_i^{(m)}(t)$.
\begin{proposition}\label{sel}
There exists a subsequence $\{m_j\}_{j\in\mathbb N}
\subset\mathbb N$ and a family of Radon measures
$\{\hat\mu_t\}_{t\geq 0}$ such that $\hat\mu_t=
\lim_{j\rightarrow\infty}\|V_t^{(m_j)}\|$ as measures for all $t\geq 0$ . 
\end{proposition}
\begin{proof} 
If we can prove the existence of such a subsequence
on any $[0,T]$, then a diagonal argument shows the
existence of a subsequence with the desired
property. Thus we fix $T>0$ in the following. 
By Proposition \ref{lengthcontrol} with 
$V_t^{(m)}$, we have \eqref{goo3}-\eqref{gaa1} with $\Cr{c1}=\Cr{c1}(u^{(m)},T)$. Then, by \eqref{ucon1} 
and \eqref{uconin}, the right-hand sides of 
\eqref{goo3}-\eqref{gaa1} are bounded depending only on $\|V_0\|(\R^2)=\mathcal H^1(\Gamma_0)$,
$\int_0^{T}\|\nabla u(\cdot,t)\|_{L^2(\mathbb R^2)}^2\,dt$ and $\esssup_{t\in[0,T+1]}\|u(\cdot,t)\|_{L^2(\mathbb R^2)}$ for all large $m$. Thus, there exists a
constant $\Cl[con]{ct}$ depending only on these
quantities (and 
not on $m$) such that
\begin{equation}\label{gmass}
   \sup_{t\in[0,T]} \|V_t^{(m)}\|(\mathbb R^2)\leq \Cr{ct}.
\end{equation}
Then  arguing precisely as in the proof of 
\eqref{goo4}, we may obtain
\begin{equation}\label{Bra1}
 \|V_t^{(m)}\|(\mathbb R^2)\big|_{t=t_1}^{t_2}
    + \frac 14\int_{t_1}^{t_2}\int_{\mathbb R^2} |h(V_t^{(m)})|^2\,d\|V_t^{(m)}\|dt\leq  \Cl[con]{ct2}\int_{t_1}^{t_2}\int_{\mathbb R^2}|\nabla u^{(m)}|^2\,dxdt,
\end{equation}
where $\Cr{ct2}$ is independent of $m$. 
To proceed, we introduce the following quantities: 
\begin{subequations}
\begin{align}
&\Phi^{(m)}(t):=\|V_t^{(m)}\|(\mathbb R^2),\\
    &H^{(m)}(t):=\int_{0}^t
    \int_{\mathbb R^2}\frac14 |h(V_s^{(m)})|^2\,d\|V_s^{(m)}\|ds,\\
    &U^{(m)}(t):=\Cr{ct2}\int_{0}^t \int_{\mathbb R^2}|\nabla u^{(m)}(x,s)|^2\,dxds, \quad U(t):=\Cr{ct2}\int_{0}^t\int_{\mathbb R^2}|\nabla u(x,s)|^2\,dxds, \\ &
    \Psi^{(m)}_1:=\Phi^{(m)}+H^{(m)}-U^{(m)},\,\,\Psi^{(m)}_2:=\Phi^{(m)}-U^{(m)}.
\end{align}
\end{subequations}
Then \eqref{Bra1} shows that 
$\Psi^{(m)}_1$ and $\Psi^{(m)}_2$ are non-increasing
functions of $t$. For a sequence of 
uniformly bounded non-increasing functions, we may extract a subsequence $\{m_j\}_{j\in\mathbb N}$ such that $\Psi_1^{(m_j)}$ and
$\Psi_2^{(m_j)}$ converge to some 
non-increasing functions point-wise.
More precisely, by a diagonal argument, we may choose a subsequence that converges point-wise on   $\mathbb Q_+=\mathbb Q\cap [0,\infty)$ first. The limit function is again a non-increasing function on $\mathbb Q_+$, which 
can be uniquely extended to a 
continuous function on $[0,\infty)$ except for a set of countable discontinuous points named $D\subset[0,\infty)$. Using again the non-increasing property, one can show the point-wise convergence on $[0,\infty)\setminus D$, and one may
choose a further subsequence to have
a convergence on countable $D$ by
the diagonal argument. This gives a desired subsequence $\{m_j\}_{j\in\mathbb N}$.
Since $U^{(m_j)}$ converges to $U$ 
point-wise by \eqref{ucon1}, we may conclude that
both $\Phi^{(m_j)}$ and $H^{(m_j)}$ converge point-wise in $t$, and we define
\begin{equation}\label{con1}
\Phi:=\lim_{j\rightarrow\infty}\Phi^{(m_j)}\text{ and }H:=\lim_{j\rightarrow\infty}H^{(m_j)}.
 \end{equation}
Note that a non-increasing function is 
continuous on a co-countable set.
From this and also the fact that $U$ is a continuous function, we conclude that
$H$ and $\Phi$ are continuous on a 
co-countable set. 
By the compactness property of Radon measures, on a countable dense set $\{t(k)\}_{k\in\mathbb N}$ of $[0,T]$, we may further extract a subsequence (denoted by the same notation) such that 
\begin{equation}\label{con2}
    \|V_{t(k)}^{(m_j)}\|\text{  converges to 
some Radon measure } \mu_{t(k)}\text{ as }
j\rightarrow\infty\text{ for each }k\in\mathbb N.
\end{equation}

We next prove that $\mu_{t(k)}$ can 
be extended to a Radon measure $\hat\mu_t$ 
on a co-countable set $A\subset[0,T]$ continuously and that
$\|V_t^{(m_j)}\|$ converges to $\hat\mu_t$ for all $t\in A$. 
For any $t_2>t_1$ and for any $\phi\in C^2_c(\mathbb R^2;\mathbb R_+)$, we use 
\eqref{bra-n} and the Cauchy-Schwarz inequality to obtain
\begin{equation}\|V_t^{(m)}\|(\phi)\|\big|_{t=t_1}^{t_2}\leq
\int_{t_1}^{t_2}\int_{\mathbb R^2}-\frac12 |h(V_t^{(m)})|^2\phi+\Cl[con]{coco}\frac{|\nabla\phi|^2}{\phi}+\|\phi\|_{L^\infty}|u^{(m)}|^2\,d\|V_t^{(m)}\|(x)dt.
\end{equation}
Here $\Cr{coco}$ is a constant 
independent of $m$.
Since $\sup_{\{\phi>0\}} (|\nabla \phi|^2/\phi)\leq 2\|\nabla^2\phi\|_{L^\infty}$
(see \cite[Lemma 3.1]{tonegawa2019brakke}), the second term on the right-hand side is bounded by $C(t_2-t_1)$
for some $C$ independent of $m$. Thus the same
computation in Proposition \ref{lengthcontrol} shows
    \begin{equation}\label{bra-up}
        \|V_t^{(m)}\|(\phi)\big|_{t=t_1}^{t_2}\leq \Cl[con]{ct3}\big(t+H^{(m)}(t)+U^{(m)}(t)\big)\big|_{t=t_1}^{t_2}.
    \end{equation}
     Here $\Cr{ct3}$ depends on  $\|\phi\|_{C^2}$ and others but not $m$. 
   Now let $t_2=t(k')>t(k)=t_1$ and $m_j$ in place of $m$. Since $\|V_{t(k)}^{(m_j)}\|(\phi)$ converges to $\mu_{t(k)}(\phi)$, and also $H^{(m_j)}(t(k))$ and $U^{(m_j)}(t(k))$ converge to $H(t(k))$ and $U(t(k))$ respectively, we have
   \begin{equation}
       \mu_{t}(\phi)\big|_{t=t(k)}^{t(k')}\leq 
       \Cr{ct3}(t+H(t)+U(t))\big|_{t=t(k)}^{t(k')}.
   \end{equation}
  This implies that $\mu_t(\phi)-\Cr{ct3}(t+H(t)+U(t))$ on the set $\{t(k)\}_{k\in\mathbb N}$ (recall that this set is dense in $[0,T]$) 
  is a monotone non-increasing function, thus 
  can be extended as a continuous function on a co-countable 
  set of $[0,T]$. Since $H$ is also continuous on a co-countable set
  and $U$ is continuous, we may 
  define $\hat\mu_t(\phi)=\lim_{j\rightarrow\infty}\mu_{t(k_j)}(\phi)$ (whenever $t(k_j)\rightarrow t$) on a co-countable set of $[0,T]$
  as a continuous function. The above argument was for a fixed $\phi\in C^2_c(\mathbb R^2;\mathbb R_+)$, but we can choose a countable set $\{\phi^{(\ell)}\}_{\ell\in\mathbb N}\subset C^2_c(\mathbb R^2;\mathbb R_+)$ which is dense in $C_c(\mathbb R^2;\mathbb R_+)$ with respect to the norm $\|\cdot\|_{L^{\infty}}$, and conclude that $\hat\mu_t(\phi^{(\ell)})$ is similarly defined on a co-countable set of $[0,T]$ for all $\ell\in\mathbb N$. We name this co-countable set again as $A$. Next, we prove that $\lim_{j\rightarrow \infty}\|V_t^{(m_j)}\|(\phi^{(\ell)})=
  \hat\mu_t(\phi^{(\ell)})$ for all $t\in A$.  
  To be clear, let us call an arbitrary point of $A$ as $\tilde t\in A$ and let $\{t(k_n)\}_{n\in\mathbb N}$ be any sequence with $t(k_n)<\tilde t$ and $\lim_{n\rightarrow\infty}t(k_n)=\tilde t$. 
   Since \eqref{bra-up} is valid for any $t_2>t_1$, choosing $t_2=\tilde t$ and 
   $t_1=t(k_n)$, we have
   \begin{equation}
       \|V_{\tilde t}^{(m_j)}\|(\phi^{(\ell)})\leq 
       \|V_{t(k_n)}^{(m_j)}\|(\phi^{(\ell)})+\Cr{ct3}(t+H^{(m_j)}(t)+U^{(m_j)}(t))\big|_{t=t(k_n)}^{\tilde t}.
   \end{equation}
By taking $\limsup_{j\rightarrow\infty}$
and using \eqref{con1} and \eqref{con2}, we have
\begin{equation}
\limsup_{j\rightarrow\infty}\|V_{\tilde t}^{(m_j)}\|(\phi^{(\ell)})\leq \mu_{t(k_n)}(\phi^{(\ell)})+\Cr{ct3}(t+H(t)+U(t))\big|_{t=t(k_n)}^{\tilde t},
\end{equation}
and since $H$ and $U$ are continuous at $\tilde t$ and we define
$\hat\mu_t(\phi^{(\ell)})$ as above, by letting $n\rightarrow\infty$, we have
\begin{equation}
    \limsup_{j\rightarrow\infty} \|V_{\tilde t}^{(m_j)}\|(\phi^{(\ell)})
    \leq \hat\mu_{\tilde t} (\phi^{(\ell)}). 
\end{equation}
Similarly, let $\{t(k_n)\}_{n\in\mathbb N}$ be any sequence with $\tilde t<t(k_n)$ and $\lim_{n\rightarrow\infty} t(k_n)=\tilde t$. Again letting $t_2=t(t_n)$ and $t_1=\tilde t$, we have
\begin{equation}
   \|V_{t(k_n)}^{(m_j)}\|(\phi^{(\ell)})-\Cr{ct3}(t+H^{(m_j)}(t)+U^{(m_j)}(t))\big|^{t(k_n)}_{t=\tilde t}\leq 
       \|V_{\tilde t}^{(m_j)}\|(\phi^{(\ell)}). 
\end{equation}
By arguing similarly and taking $\liminf_{n\rightarrow \infty}$, we have
\begin{equation}
    \hat\mu_{\tilde t}(\phi^{(\ell)})\leq \liminf_{j\rightarrow\infty}
    \|V_{\tilde t}^{(m_j)}\|(\phi^{(\ell)}).
\end{equation}
This proves the desired convergence on $A$, and since $\{\phi^{(\ell)}\}_{\ell\in\mathbb N}$ is dense in $C_c(\mathbb R^2;\mathbb R_+)$, we have the convergence of $\|V_{\tilde t}^{(m_j)}\|$
to $\hat\mu_{\tilde t}$ as Radon measures. On the complement of $A$ which is countable, we may choose a further subsequence (denoted by the same notation) of $\{m_j\}$ so
that $\|V_t^{(m_j)}\|$ converges to some Radon measure, which we again name $\hat\mu_t$. This
ends the proof. 
\end{proof}
\begin{proposition}\label{l2}
   For a.e.~$t\in[0,\infty)$, there exists a unique $V_t\in{\bf IV}_1(\mathbb R^2)$ such that $\hat \mu_t=\|V_t\|$, $\|\delta V_t\|$ is absolutely
   continuous with respect to $\|V_t\|$ and $h(V_t)$
   is in $L^2$. Moreover, for any $T<\infty$, \eqref{h2} holds
   for $h(V_t)$. 
   
\end{proposition}
\begin{proof} 
Taking $\liminf$ in \eqref{Bra1} with $m=m_j$, $t_1=0$, $t_2=T$ and using Fatou's lemma, we have  
\begin{equation}\label{tonegawabookequ3.8} 
\liminf_{j \rightarrow \infty} \int_{\R^2}\left|h\left(V_s^{(m_j)}\right)\right|^2 \,d\|V_s^{(m_j)}\|<\infty\text{ for } a.e. ~s \in[0, T].
\end{equation}
On the other hand, we have  
\begin{align}
\|\delta V_s^{(m_j)}\|(\mathbb R^2)&= \int_{\mathbb R^2}|h(V_s^{(m_j)})| \,d\|V_s^{(m_j)}\| \nonumber\\
&\leq\left(\int_{\mathbb R^2}|h(V_s^{(m_j)})|^2 \,d\|V^{(m_j)}_s\|\right)^{\frac{1}{2}}(\|V_s^{(m_j)}\|(\mathbb R^2))^{\frac{1}{2}}.
\end{align}
With these two assertions and \eqref{gmass}, we can apply     Allard's compactness theorem for  integral varifolds \cite[Sec.~6.4]{allard1972first} to a.e.~$s \in[0, T]$ so that there exists a subsequence $ \{m_j^{\prime} \} \subset\left\{m_j\right\}$ and $V_s \in \mathbf{IV}_1(\R^2)$ such that $  V_s^{(m_j^{\prime})}\xrightarrow{j \rightarrow \infty}V_s$ as varifolds. Note that the choice of the subsequence $m_j'$ may depend on $s$. On the other hand, we know that $  \|V_s^{(m_j)} \|\xrightarrow{j \to\infty} \hat{\mu}_s$ as 
measures on $\mathbb R^2$, so we have $\hat{\mu}_s=\|V_s\|$. For an integral 1-varifold $V_s$, there exist a countably $1$-rectifiable set $M_s \subset \mathbb{R}^2$ and a function $\theta_s: M_s \rightarrow \mathbb{N}$ such that $V_s={\bf var}(M_s,\theta_s)$ and $\|V_s\|=\theta_s \mathcal{H}^1\lfloor_{M_s}$, which is $\hat{\mu}_s$. 
Once we know that $\hat{\mu}_s$ is of this form, then $\hat{\mu}_s$ uniquely determines $V_s$. So for a.e. $s \in[0, T]$, we have a unique $V_s \in \mathbf{I V}_1(\R^2)$ such that $\|V_s\|=\hat\mu_s$. We define $V_s \in \mathbf{V}_1(\R^2)$ for other $s \in[0, T]$ so that $\|V_s\|=\hat{\mu}_s$, for example, $V_s(\phi)=\int_{\R^2} \phi (x, S_0 ) \,d \hat{\mu}_s(x)$ for $\phi \in C_c (G_1(\R^2) )$,
where $S_0 \in \mathbf{G}(2, 1)$ is a fixed element. \footnote{Note that this is not a rectifiable varifold. However,  the measure of such a set of times is zero.} The resulting family $\left\{V_s\right\}_{s \in[0, T]}$ belongs to $\mathbf{IV}_1(\R^2)$ for a.e.~$s \in[0, T]$ and satisfies the claimed convergence properties.
The existence of $h(V_s)\in L^2(\|V_s\|)$ follows from the 
standard argument: By the varifold convergence,
one can prove that, for any $g\in C_c^1(\mathbb R^2;\mathbb R^2)$,
\begin{equation}\label{hint}
\begin{split}
    \delta V_s (g)&=\lim_{j\rightarrow\infty}
    \delta V_s^{(m_j')}(g)\leq \lim_{j\rightarrow\infty}\left(\int_{\mathbb R^2}|h(V_s^{(m_j')})|^2 \,d\|V^{(m_j')}_s\|\right)^{\frac{1}{2}}\left(\int_{\mathbb R^2}|g|^2\,d\|V_s^{(m_j')}\|\right)^{\frac{1}{2}}\\
&=\left(\liminf_{j\rightarrow\infty}\int_{\mathbb R^2}|h(V_s^{(m_j)})|^2 \,d\|V^{(m_j)}_s\|\right)^{\frac{1}{2}}\left(\int_{\mathbb R^2}|g|^2\,d\|V_s\|\right)^{\frac{1}{2}}.
    \end{split}
\end{equation}
In the last line, we have taken   a subsequence $\{m_j'\}\subset\{m_j\}$ so that the limit of $\int |h(V_s^{(m_j')})|^2\,d\|V_s^{(m_j')}\|$ achieves the $\liminf$ (left-hand side of \eqref{tonegawabookequ3.8}). The inequality \eqref{hint} shows the absolute continuity of the total
variation measure $\|\delta V_s\|$, which shows 
the existence of $h(V_s)\in L^2(\|V_s\|)$. By integrating 
\eqref{hint} in time, by Fatou's lemma and \eqref{h2} satisfied
by $V_s^{(m)}$,
we also have \eqref{h2} for
$h(V_s)$ as claimed.
\end{proof}
\begin{proposition}\label{junct}
The property of (V3) holds for $V_t$. 
\end{proposition}
\begin{proof}
For a.e.~$s$, since $h(V_s)\in L^2(\|V_s\|)$, we have the monotonicity formula (see \cite[17.7]{simon1983lectures},\cite[8.1]{allard1972first}) 
at each point of ${\rm spt}\,\|V_s\|$ and
there exist stationary tangent cones (see \cite{simon1983lectures} for the definition of tangent cone and the basic properties).
Note that each tangent cone consists 
of half-lines with possible integer multiplicities.
Thus
either of the following two cases holds:
\begin{itemize}
\item[(a)] There is at least one tangent cone consisting of a 
single line with integer multiplicity. 
\item[(b)] 
Any tangent cone is a union of half-lines
emanating from the origin (whose union is not a single line). 
\end{itemize}
Consider the case (a) first, and without loss of generality, assume that the point in question is the origin $x=0$. If the density is equal to one, then 
the Allard regularity theorem shows that
${\rm spt}\|V_s\|$ can be represented by
a $C^{1,1/2}$ curve with the multiplicity equal to one 
in a neighborhood of the origin. Since the
curvature is in $L^2$, we immediately deduce that
it is a $W^{2,2}$ curve. Thus we may assume 
that the density at $0$ is an integer larger than one in the 
following. Let $V_s^{(m_j')}$ be the subsequence in the proof of Proposition \ref{l2} with uniform $L^2$ bound of $h(V_s^{(m_j')})$. Let $\{r_i\}$ be a sequence converging to $0$ and such that the rescaling of
$V_s$ by $1/r_i$ converges to a line $L$ with 
integer multiplicity. We claim that for all 
sufficiently large $i$ and $j$, $B_{r_i}$ does
not contain any junction of $V_s^{(m_j')}$.
Here, a ``junction'' means a location where curves 
meet with non-zero angles. 
For a contradiction, assume otherwise. Then, 
there must exist subsequences (without changing
the indices) such that $B_{r_i}$ contains at least 
one junction of $V_s^{(m_j')}$ and we may also
choose a further subsequence so that the Hausdorff
distance between $L$ (the limit tangent line) and ${\rm spt}\|V_s^{(m_j')}\|$ in $B_{2r_i}$ is $o(r_i)$ . 
Then consider the change
of variables $\tilde x=x/r_i$ and the corresponding
varifolds under this change of variable as $\tilde V_s^{(m_j')}$. We then have
\begin{equation*}
\int_{B_2}|h(\tilde V_s^{(m_j')})|^2\,d\|\tilde V_s^{(m_j')}\|=
r_i\int_{B_{2r_i}}|h(V_s^{(m_j')})|^2\,d\|V_s^{(m_j')}\|\leq 
r_i\int_{\mathbb R^2}|h(V_s^{(m_j')})|^2\,d\|V_s^{(m_j')}\|
\end{equation*}
and the right-hand side converges to $0$ as 
$i\rightarrow \infty$. Thus, for all sufficiently
large $i$, the $W^{2,2}$ curves consisting of
${\rm spt}\|V_s^{(m_j')}\|$ are close to straight line 
segments in $C^1$ topology. On the other hand, at a junction
in $B_1$, there must be at least two curves
emanating from there having an angle of
$120$ degrees. One can draw a conclusion that 
such a set of curves cannot be contained in 
a small neighborhood of a line
for all sufficiently large $i$, and we 
arrive at a contradiction. Thus, for some 
$r_i$ and all sufficiently large $j$, $V_s^{(m_j')}$ is free of junctions
in $B_{r_i}$ with small $L^2$ curvatures,
and can be represented as a union of graphs
of functions
$f_1\leq \ldots\leq f_l$ in $W^{2,2}$ close to a line. Then one can prove that the limit 
$V_s$ in $B_{r_i}$ may be 
expressed similarly. This gives a 
desired description at a point belonging to the
case (a). Next, let us consider the case (b).
First note that the points of (b) are 
discrete point. If 
not, then there must be an accumulation 
point which must be also in (b) since 
the first part of the argument shows the 
absence of a  junction in a
neighborhood if it is in (a). Let $x_i\rightarrow \hat x$ be a sequence of converging points
of (b). Then, let $r_i:=|x_i-\hat x|$ and consider the rescaling of $V_s$ centered at $\hat x$ by $1/r_i$. By the 
definition of (b), it converges to
some tangent cone which is not a line. On the
other hand, in a small neighborhood of $\tilde x:=\lim_{i\rightarrow\infty}(x_i-\hat x)/|x_i-\hat x|$ (after choosing a convergent  subsequence), the 
support of the rescaled $V_s$ converges to a line since the
limit is a sum of half lines.
Since the rescaled $V_s^{(m_j')}$ is converging
and is close to a line, the first part of the
argument shows that the rescaled 
$V_s^{(m_j')}$ in a small 
neighborhood of $\tilde x$ is
free from   junctions for all sufficiently
large $j$, thus $V_s$ near $x_i$ is 
a point with the same description in the case of (a). This is a contradiction, so the points of (b)
are discrete. At such a point, there exists 
a neighborhood in which the point belonging to (b) does not exist except at $x=0$, and ${\rm spt}\|V_s\|\setminus\{0\} $ 
consists of points in (a). By the description of (a) and choosing a small neighborhood, one can argue that ${\rm spt}\|V_s\|$ is
represented as a finite number of $W^{2,2}$ curves
reaching to the junction point with small $C^{1}$ variations (note that the
curvature bound gives $C^{1,1/2}$ 
control of the curves). The angles of intersection of these curves
must be either $0$, $60$ or $120$ degrees
due to the fact that $V_s^{(m_j')}$ is converging
with the same property of junctions and in the uniformly bounded 
$C^{1,1/2}$ norms of curves controlled. For $N=2$,
since $V_s^{(m_j')}$ has no junction, the $V_s$
cannot have any point of (b). This completes the
proof.
\end{proof}
\begin{proposition}
    The function $u$ satisfies
    for any $T<\infty$
    \begin{equation}\label{uapp1}
\lim_{j\rightarrow\infty}\int_0^T\int_{\R^2}|u^{(m_j)}-u|^2\,d\|V_t\|dt=
\lim_{j\rightarrow\infty}\int_0^T\int_{\R^2}|u^{(m_j)}-u|^2\,d\|V^{(m_j)}_t\|dt=0
\end{equation}
and with $\Cr{c1}(u,T)$ as in \eqref{NS1}, we have
    \eqref{gaa1} satisfied for 
    $u$, i.e., 
     \begin{equation*}
\int_{0}^{T}\int_{\R^2}|u|^2\,d\|V_t\|dt \leq 4\sqrt{\Cr{c1}\Cr{MZ}}\|V_0\|(\R^2)\sqrt{1+\Cr{MZ}\Cr{c1}\exp(\Cr{MZ}\Cr{c1})}\exp(\Cr{MZ}\Cr{c1}/2).
\end{equation*}
\end{proposition}
\begin{proof}
For \eqref{uapp1}, first note that Theorem
\ref{MZ} holds true equally for $\phi\in W^{1,2}(\R^{n+1})$ by approximation. Thus, computing as 
in \eqref{u2Vt} and using \eqref{lenest2} and \eqref{uconin},
\begin{equation*}
\begin{split}
    &\int_{0}^{T}\int_{\R^2}|u^{(m_j)}-u|^2\,d\|V_t\|dt\leq \sqrt{\Cr{MZ}}\int_{0}^{T}\mathfrak{D}(\|V_t\|)\int_{\R^2}|\nabla |u^{(m_j)}-u|^2|\,dxdt \\
    &\leq 2\sqrt{\Cr{MZ}}\int_{0}^T(\|V_t\|(\R^2))^{\frac12}\big(\int_{\R^2}|h(V_t)|^2\,d\|V_t\|\big)^{\frac12}
    \int_{\R^2} |u^{(m_j)}-u||\nabla(u^{(m_j)}-u)|\,dxdt
    \end{split}\end{equation*}
    \begin{equation}
        \begin{split}
    &\leq 4\sqrt{\Cr{MZ}}(\sup_{t\in[0,T]}\|V_t\|(\R^2))^{\frac12}\sup_{t\in[0,T+1]}\|u(\cdot,t)\|_{L^2(\R^2)} \big(\int_0^T\int_{\R^2}|h(V_t)|^2\,d\|V_t\|dt\big)^{\frac12}\\ 
    &\hspace{1cm}\big(\int_{0}^T\int_{\R^2}|\nabla (u^{(m_j)}-u)|^2\,dxdt\big)^{\frac12}.
    \end{split}\label{l2con}
\end{equation}
By \eqref{ucon1}, this converges to $0$
as $j\rightarrow \infty$. The same 
computation with respect to $V_t^{(m_j)}$ 
holds equally, thus we obtain \eqref{uapp1}.
Since $u^{(m_j)}$ satisfies \eqref{gaa1}, 
the same inequality holds for $u$ due to the convergence
of $\|V_t^{(m_j)}\|\rightarrow
\|V_t\|$.
\end{proof}
\begin{proposition}
  $\{V_t\}_{t\geq 0}$ satisfies the
  inequality \eqref{bra}. 
\end{proposition}
\begin{proof} We fix $T>0$,
$0\leq t_1<t_2\leq T$ and $\phi\in C_c^2(\R^2\times[0,T])$ with $\phi\geq 0$. Once this case is treated, the 
$C_c^1$ case can be proved by 
approximation. 
Let $\{m_j\}$ be as in Proposition 
\ref{sel}. 
Given $\varepsilon>0$, we fix $j_0$ so that
\begin{equation}\label{gerror}
    \int_0^T\int_{\R^2}|u^{(m_j)}-u|^2\,d\|V_t\|dt<\varepsilon
\end{equation}
when $j\geq j_0$ by using \eqref{uapp1}.
For notational simplicity, set $g=u^{(m_{j_0})}$. Now fix $s\in[0,T]$
such that \eqref{tonegawabookequ3.8} holds, and 
let a subsequence $ \{m_j^{\prime} \}_{j \in \mathbb{N}}$ be such that   
\begin{subequations}
\begin{align}
&V_s^{(m_j^{\prime})} \rightarrow V_s\text{ as varifolds},\label{mjprime}\\
&\lim_{j \rightarrow \infty} \int_{\R^2} \phi|h(V_s^{(m_j')})|^2-\nabla\phi\cdot g^\perp-h(V_s^{(m_j')}) \cdot (\nabla \phi-g \phi)\,d\|V_s^{(m_j')}\|\nonumber\\
=&\liminf_{j \rightarrow \infty} \int_{\R^2}\phi|h(V_s^{(m_j)})|^2-\nabla\phi\cdot g^\perp-h(V_s^{(m_j)}) \cdot (\nabla \phi-g \phi) \,d\|V_s^{(m_j)}\|.\label{achievesliminf}
\end{align}
\end{subequations}
Note that the choice of subsequence may depend on $s$. 
Using \eqref{tonegawakimequ2.2a}, \eqref{tonegawakimequ2.2} and \eqref{mjprime}, we can deduce
\begin{equation}\label{achievesliminf1}
\begin{aligned}
\lim_{j \rightarrow \infty} &\int_{\R^2}\nabla\phi\cdot g^\perp \, d\|V_s^{(m_j')}\|+\lim_{j \rightarrow \infty} \int_{\R^2}  h(V_s^{(m_j')}) \cdot (\nabla \phi-g \phi)\,d\|V_s^{(m_j')}\|\\
=&\int_{\R^2}\nabla\phi\cdot g^\perp \, d\|V_s\| -  \int_{G_1(\R^2)} \nabla(\nabla \phi-g \phi)\cdot S\,dV_s.
\end{aligned}
\end{equation}
By the similar argument for \eqref{hint}, 
we have
$$\int_{\R^2} \phi|h(V_s)|^2  \,d \|V_s \| \leq \lim_{j \rightarrow \infty} \int_{\R^2}\phi|h (V_s^{(m_j')})|^2  \,d \|V_s^{(m_j')}\|,$$
thus with \eqref{achievesliminf} and \eqref{achievesliminf1},
we find
\begin{equation}
\begin{split}
\int_{\R^2} &\phi|h(V_s)|^2-\nabla\phi\cdot g^\perp-h(V_s) \cdot (\nabla \phi-g \phi) \,d \|V_s \| \\
\leq & \liminf_{j \rightarrow \infty} \int_{\R^2}\phi|h (V_s^{(m_j)} )|^2-\nabla\phi\cdot g^\perp-h(V_s^{(m_j)}) \cdot (\nabla \phi-g \phi)   \,d \|V_s^{(m_j)} \|.
\end{split}
\end{equation}
Since $\phi|h|^2-|\nabla\phi||g|-|h|(|\nabla \phi|+|g|\phi)\geq -|\nabla\phi||g|-\frac{|\nabla\phi|^2}{\phi}-|g|^2\phi$ which is bounded from below by a constant depending only on $g$ and $\phi$, we may use Fatou's lemma 
to conclude that 
\begin{equation}\label{tonegawabookequ3.9}
\begin{aligned}
&\int_{t_1}^{t_2}\int_{\R^2} \phi|h(V_s)|^2-\nabla\phi\cdot g^\perp-h(V_s) \cdot (\nabla \phi-g \phi) \,d \|V_s \|ds \\
&\leq \liminf_{j \rightarrow \infty}\int_{t_1}^{t_2} \int_{\R^2} \phi|h(V_s^{(m_j)})|^2-\nabla\phi\cdot g^\perp-h(V_s^{(m_j)}) \cdot (\nabla \phi-g \phi) \,d\|V_s^{(m_j)}\|ds.
\end{aligned}
\end{equation}
Since each $V^{(m_j)}$ satisfies
\eqref{bra-n} with $u^{(m_j)}$, 
\begin{equation}\label{tonegawabookequ3.11}
\begin{aligned}
& \liminf _{j \rightarrow \infty} \int_{t_1}^{t_2} \int_{\R^2} \phi|h(V_s^{(m_j)})|^2-\nabla\phi\cdot g^\perp-h(V_s^{(m_j)}) \cdot (\nabla \phi-g \phi)\, d\|V_s^{(m_j)}\|ds \\
&
= \liminf _{j \rightarrow \infty}\big( \|V_{t_1}^{(m_j)} \|\left(\phi\left(\cdot, t_1\right)\right)- \|V_{t_2}^{(m_j)} \|\left(\phi\left(\cdot, t_2\right)\right)+\int_{t_1}^{t_2} \int_{\R^2} \partial_t\phi \,d\|V_s^{(m_j)}\|ds  \\
& \quad\quad +\int_{t_1}^{t_2} \int_{\R^2}    \big(\nabla \phi \cdot (u^{(m_j)}-g)^\perp+\phi\, h(V_s^{(m_j)}) \cdot (g-u^{(m_j)})\big)\,d\|V_s^{(m_j)}\|ds \big)\\
& =\left\|V_{t_1}\right\|\left(\phi\left(\cdot, t_1\right)\right)-\left\|V_{t_2}\right\|\left(\phi\left(\cdot, t_2\right)\right)+\int_{t_1}^{t_2} \int_{\R^2} \partial_t\phi \,d \|V_s \| d s\\
&\qquad + \liminf _{j \rightarrow \infty}\int_{t_1}^{t_2} \int_{\R^2}    \Big(\nabla \phi \cdot (u^{(m_j)}-g)^\perp+\phi\, h(V_s^{(m_j)}) \cdot (g-u^{(m_j)})\Big)\,d\|V_s^{(m_j)}\| ds\end{aligned}
\end{equation}
where in the last step we have  used    $\|V_s^{(m_j)}\|\to  \|V_s \|$ as Radon measures for all $s$. For the integrals in the last line of \eqref{tonegawabookequ3.11}, we may use \eqref{uapp1} and \eqref{gerror} to conclude that they 
are bounded by a multiple of $\varepsilon$. Similarly, we may replace $g$ on the first line of \eqref{tonegawabookequ3.9} by $u$ with 
a similar error. By letting
$\varepsilon\rightarrow 0$, this gives the desired
inequality \eqref{bra} for $V_t$.
\end{proof}
\begin{remark}
 For the family of varifolds 
 $\{V_t\}_{t\geq 0}$ thus far 
 constructed, from (V1') and from the method of
 construction, (V1) is satisfied. 
 We have proved so far (V3)-(V7)
 and (V2) will be proved in
 Remark \ref{re1}.
 \end{remark}
\begin{proposition}
There exists a subsequence of $\{m_j\}$ (we use the same notation
for this subsequence) with the following property:
    For each $i=1,\ldots,N$, there
    exists a set $S(i)\subset\R^2\times[0,\infty)$
    with $\|\nabla \chi_{S(i)}\|(\R^2\times[0,T])<\infty$ for 
    all $T>0$ such that,
    writing $E_i(t):=\{x\in\R^2 : (x,t)\in S(i)\}$,
    \begin{equation}\label{concha}
\lim_{j\rightarrow\infty}\int_{B_R}|\chi_{E^{(m_j)}_i(t)}(x)-
        \chi_{E_i(t)}(x)|\,dx=0
    \end{equation}
    for all $R>0$ and for all $t\geq 0$. Moreover, (E5) holds true.
\end{proposition}
\begin{proof} We fix $i=1,\ldots,N$ and $T>0$ in the following. 
    Recall the definition $E_i^{(m)}(t)$ being the 
    ``$i$-th phase'' corresponding
    to the flow with $u^{(m)}$
    obtained by Theorem \ref{exreg}.
    Let $S^{(m)}(i)$ be the 
    set defined in (E6') corresponding to $E^{(m)}(t)$, 
    namely, 
    \begin{equation}
        S^{(m)}(i):=\{(x,t) : x\in E_i^{(m)}(t), t\geq 0\}
    \end{equation}
which is a relative open set in $\R^2\times[0,\infty)$. 
For any $g=(g^1,g^2,g^3)\in C_c^1(\R^2\times (0,T);\R^3)$, (E7') with $t_1=0$ and $t_2=T$ shows
\begin{equation}
\begin{split}
&\int_{S^{(m)}(i)}g^1_{x_1}+g^2_{x_2}+\frac{\partial g^3}{\partial t}\,dxdt=\int_0^T\int_{E^{(m)}_i(t)}
g^1_{x_1}+g^2_{x_2}+\frac{\partial g^3}{\partial t}\,dxdt\\ 
&=\int_0^T \int_{\partial^* E^{(m)}_i(t)}(g^1,g^2)\cdot \nu_i^{(m)}
-(h(V_t^{(m)})+u^{(m)})\cdot \nu_i^{(m)} g^3\,d\mathcal H^1 dt,
\end{split}
\end{equation}
where $\nu_i^{(m)}$ is the outer
unit normal to $\partial^* E_i^{(m)}(t)$. 
By (E5'), 
for all $0\leq t\leq T$,
\begin{equation}\label{esthu1}
    \mathcal H^1(\partial^* E_i^{(m)}(t))= \|\nabla \chi_{E_i^{(m)}(t)}\|(\R^2)
    \leq \|V_t^{(m)}\|(\R^2)\leq
    \sup_{s\in[0,T]}\|V_s^{(m)}\|(\R^2),
\end{equation}
and again by (E5'), 
\begin{equation}\label{esthu2}
\begin{split}
    \int_0^T\int_{\partial^* E_i^{(m)}(t)}&|h(V_t^{(m)})|+|u^{(m)}|\,d\mathcal H^1 dt\leq \int_0^T\int_{\R^2}|h(V_t^{(m)})|+|u^{(m)}|\,d\|V_t^{(m)}\|dt \\
    &\leq (T\sup_{s\in[0,T]}\|V_s^{(m)}\|(\R^2))^{\frac12}\big(\int_{0}^T\int_{\R^2}|h(V_t^{(m)})|^2+|u^{(m)}|^2\,d\|V_t^{(m)}\|dt\big)^{\frac12}.
\end{split}
\end{equation}
Since the right-hand sides of \eqref{esthu1} and \eqref{esthu2} are both bounded uniformly due to \eqref{goo3}-\eqref{gaa1},
this implies that $\|\nabla \chi_{S^{(m)}(i)}\|(\R^2\times (0,T))$ is uniformly bounded. 
By the compactness theorem for
bounded variation function applied to
the subsequence $\{m_j\}$ in 
Proposition \ref{sel}, we have a 
further subsequence (denoted by the 
same index) and a set of finite 
perimeter $S(i)\subset \R^2\times(0,T)$ such that $\chi_{S^{(m_j)}(i)}$ converges
to $\chi_{S(i)}$ locally in $L^1(\R^2\times(0,T))$. By Fubini
Theorem, we also may assume that 
$\chi_{E_i^{(m_j)}(t)}$ converges
locally to $\chi_{E_i(t)}$ in $L^1
(\R^2)$ and for 
a.e.~$t\in (0,T)$. In fact, since
(E7') holds equally true for 
$\phi\in C_c(\R^2)$
with $|\phi|\leq 1$, for
$0\leq t_1<t_2\leq T$, we have
\begin{equation}\begin{split}
    \big|\int_{E_i^{(m)}(t_2)}\phi(x)\,dx& -\int_{E_i^{(m)}(t_1)}\phi(x)\,dx\big|\leq 
\int_{t_1}^{t_2}\int_{\partial^* E_i^{(m)}(t)}|h(V_t^{(m)}|+|u^{(m)}|\,d\mathcal H^1 dt \\
 &\leq \Cr{vol2}(t_2-t_1)^{\frac12}
 \end{split}
\end{equation}
where $\Cr{vol2}$ can be obtained as in
\eqref{esthu2} and is independent
of $m$. For any $R>0$, we may let $A=B_R\cap E_i^{(m)}(t_2)\setminus E_i^{(m)}(t_1)$ or $B_R\cap E_i^{(m)}(t_1)\setminus E_i^{(m)}(t_2)$ 
and approximate $\chi_A$ by 
$\phi\in C_c(\R^2)$ to see that (after letting $\phi\rightarrow \chi_A$ and $R\rightarrow \infty$) $|E_i^{(m)}(t_1)\triangle E_i^{(m)}(t_2)|\leq 2\Cr{vol2}(t_2-t_1)^{\frac12}$.  From this, one can conclude
that $\chi_{E_i^{(m_j)}(t)}$
converges locally in $L^1(\R^2)$ 
for all $t\in(0,T)$ 
(not only a.e.~$t$), and the limit
$E_i(t)$ is $\frac12$-H\"{o}lder 
continuous with respect to the
Lebesgue measure. If necessary,
we may define the time-slice of 
$S(i)$ to be $E_i(t)$ for all $t\in(0,T)$, and by a diagonal argument as $T\rightarrow \infty$, this proves
the claim of \eqref{concha} and also
(E5). 
\end{proof}
\begin{remark}\label{re1}
    By (E2'), $E_1^{(m_j)}(t),\ldots,E_N^{(m_j)}(t)$ are
    mutually disjoint for each $t\geq 0$. By \eqref{concha}, 
    it follows that $E_1(t),\ldots,E_N(t)$ are 
    mutually disjoint with respect 
    to the Lebesgue measure.
    Also since $\|\nabla\chi_{E_i^{(m_j)}(t)}\|\leq \|V_t^{(m_j)}\|$
 (by (E5')),  $\lim_{j\rightarrow\infty}\|V_t^{(m_j)}\|=\|V_t\|$, and by the lower-semicontinuity property, we have
 $\|\nabla\chi_{E_i(t)}\|\leq \|V_t\|$ for all $t\geq 0$.
 Then, by \cite[Proposition 29.4]{maggi2012sets}, 
 \begin{equation}
     \frac12\sum_{i=1}^N\|\nabla\chi_{E_i(t)}\|=\mathcal H^1\lfloor_{\cup_{i=1}^N\partial^* E_i(t)}\leq \|V_t\|.
 \end{equation}
 At $t=0$, we may argue just as in \cite[Proposition 6.10]{ST19}. If $\mathcal H^1(\cup_{i=1}^N
 \partial E_{0,i}\setminus \partial^* E_{0,i})=0$, then
 \begin{equation}
2\|V_0\|=\sum_{i=1}^N\|\nabla\chi_{E_{0,i}}\|\leq \liminf_{t\rightarrow 0+}\sum_{i=1}^N \|\nabla\chi_{E_i(t)}\|\leq\liminf_{t\rightarrow 0+}2\|V_t\|
\end{equation}
while $\limsup_{t\rightarrow 0+}\|V_t\|\leq \|V_0\|$ follows
from \eqref{bra}. Thus
$\|V_0\|=\lim_{t\rightarrow 0+}\|V_t\|$ and we have
(V2).
 \end{remark}
\begin{proposition}
    Properties (E6) holds true.
\end{proposition}
\begin{proof}
 Consider a sequence of $L^1_{loc}([0,\infty))$ functions
\begin{equation}
f^{(m_j)}(t):=2\int_{\partial^* E_i^{(m_j)}(t)}|h(V_t^{(m_j)})|^2+|u^{(m_j)}|^2\,d\|V_t^{(m_j)}\|(x)
\end{equation}
whose $L^1$-norm on $[0,T]$ is
bounded uniformly. By the compactness theorem of Radon measure,
there exists a further subsequence (denoted by the same index) and a
Radon measure $\beta$ on 
$[0,\infty)$ such that $f^{(m_j)}\,dt$ converges to $\beta$ as measures. 
With this notation, for any $\phi\in C_c^1(\R^2\times[0,\infty))$,
\begin{equation}
\int_{t_1}^{t_2}\int_{\partial^* E_i^{(m_j)}(t)}(|h(V_t^{(m_j)})|+|u^{(m_j)}|)\phi\leq \big(\int_{t_1}^{t_2}f^{(m_j)}\,dt\big)^{\frac12}\big(\int_{t_1}^{t_2}
\int_{\R^2}\phi^2\,d\|V_t^{(m_j)}\|\big)^{\frac12}
\end{equation}
where we used (E5'). Then, using 
\eqref{gvel} and the convergence of 
$f^{(m_j)}dt$, $E_i^{(m_j)}(t)$ and $\|V_t^{(m_j)}\|$ to $\beta$,
$E_i(t)$ and $\|V_t\|$ respectively, we obtain
\eqref{gvel2} and this ends the proof. 
\end{proof}
Up to this point, (E1)-(E6) are
proved. Finally,
\begin{proposition}
    The properties (E7) and (E8) hold.
\end{proposition}
\begin{proof}
This follows from the argument in
Proposition \ref{junct} and
the resulting (V3) as follows.
For a.e.~$t>0$, $V_t=\sum_{i=1}^{P(t)}{\bf var}(\ell_i(t),1)$ and away from a finite number
of junctions, the density of 
$\|V_t\|$ is precisely the number 
of $W^{1,1}$ curves passing 
through that point, all 
tangentially. As in the proof of
Proposition \ref{junct} of case (a), the 
topology of convergence of $W^{2,2}$ curves corresponding
to $V^{(m_j)}$ to the limit is locally
$C^1$, thus the properties (E8') satisfied by $V^{(m_j)}$ holds in the limit. More precisely,
suppose that $N\geq 3$ and the density of $\|V_t\|$
is $1$. Then there is only one curve going through 
this point, thus, the number of curves of $V_t^{(m_j)}$
approaching to this single curve must be also $1$.
Since $E_i^{(m_j)}(t)$ ($i=1,\ldots,N$) converges to 
$E_i(t)$, and (E8')(1) shows that the approaching curves 
are part of the reduced boundary of $E_1^{(m_j)}(t)\,\ldots,E_N^{(m_j)}(t)$, we can deduce that the
point in question must be also in the reduced boundary 
of $E_1(t),\ldots,E_N(t)$. This proves the case of $N\geq 3$. If $N=2$, then, again by the $C^1$ convergence of curves
and the property (E8')(2), one can conclude the same 
property for the limit depending on the parity of 
the density. This ends the proof of (E7). For (E8),
the unit density implies that $\theta_t(x)=1$ for 
$\|V_t\|$ a.e.~$x$ and (E7) then implies that 
$\mathcal H^1(\ell_i(t)\cap \ell_j(t))=0$ for all 
$i\neq j$ and $\mathcal H^1(\cup_{i=1}^{P(t)}\ell_i(t)
\triangle \cup_{i=1}^N \partial^* E_i(t))=0$. This immediately 
proves (E8). 
\end{proof}
\section{Final remarks}
\begin{itemize}
    \item[(1)] One may wonder if there exists a corresponding result for $n>1$ that
is also critical, for example, $q = 2$ and $p= n+ 1$ in view of \eqref{subc}, and some appropriate condition on $u$ itself, such as $\esssup_{t\in[0,T]}\int_{\mathbb R^2}|u|^2\,dx$ in the case of $n=1$. Our analysis is heavily based on
\eqref{Dtest}, and a similar strategy does not seem to work. 
\item[(2)] Suppose that we have a sequence of flows $\{V_t^{(m)}\}_{t\geq 0}$ and $\{u^{(m)}(\cdot,t)\}_{t\geq 0}$ each of which satisfies the properties of Theorem \ref{mainth1} and is  such that  $\|V_0^{(m)}\|(\mathbb R^2)$ and $\Cr{c1}(u^{(m)},T)$ are uniformly bounded for each $T>0$. 
It is then desirable to have a good compactness
property. In the case that $u^{(m)}=0$ (or even the subcritical case treated in \cite{takasao2016existence}), one can 
extract a subsequence (denoted by the same
index) 
and a limit $\{V_t\}_{t\geq 0}$ such 
that $\lim_{m\rightarrow\infty}\|V_t^{(m)}\|=\|V_t\|$ for all $t\geq 0$ and the limit is a
Brakke flow (or flow with \eqref{prob}). 
In the present case, if we can further assume
that $u^{(m)}$ converges to some $u$ strongly, namely,
\begin{equation}
    \lim_{m\rightarrow\infty}\int_0^T\int_{\mathbb R^2}|\nabla u^{(m)}-\nabla u|^2\,dxdt=0,
\end{equation}
then one can extract a subsequence such that
the limit flow is a solution of \eqref{prob}.
The proof of this may be carried out along the line of the proof of Theorem \ref{mainth1} in fact, since the proof is 
precisely the compactness of the approximate flows to the solution of \eqref{prob}. On the other hand, unless we assume this strong convergence, it is difficult to prove that the limit is a solution of \eqref{prob}. 
\end{itemize}
\appendix
\section{Proof of Theorem \ref{exreg}}\label{append}

    For a vector field $u\in C^1_c(\mathbb R^{n+1}\times[0,\infty))$, we need to change the motion law from $v=h$ to $v=h+u^\perp$ in \cite{kim2017mean,Kim-Tone2,ST-canonical} and we point out locations to be changed. First of all, let $N\in\mathbb N$ be the 
    number of phases as in 
    these papers, and with $\mathcal H^n(\Gamma_0)<\infty$, we may set the weight function  (see \cite[Section 3.1]{kim2017mean}) $\Omega\equiv 1$ and $c_1=0$. Otherwise, we may proceed with the same definitions in \cite{kim2017mean}, with 
    a slight modification provided in 
    \cite{ST-canonical} in the 
    definition of the volume-controlled Lipschitz deformation (see \cite[Definition 3.1]{ST-canonical}).
    The first major modification occurs 
    at \cite[Proposition 5.7]{kim2017mean}. The function $f$
    should be changed
    from $f(x):=x+h_{\varepsilon}(x,V)\Delta t$ to $$f(x):=x+(h_{\varepsilon}(x,V)+u(x))\Delta t$$ (where $u(x)=u(x,t)$ with $t$ specified later) and the estimates \cite[(5.55)]{kim2017mean} and \cite[(5.56)]{kim2017mean} become
    \begin{equation}\label{5.55}
        \Big|\frac{\|f_{\sharp} V\|(\phi)-\|V\|(\phi)}{\Delta t} -\delta(V,\phi)(h_{\varepsilon}(\cdot, V)+u(\cdot))\Big|\leq \varepsilon^{c_2-10},
    \end{equation}
    \begin{equation}\label{5.56}
      \frac{\|f_{\sharp} V\|(\mathbb R^{n+1})-\|V\|(\mathbb R^{n+1})}{\Delta t}+\frac12\int_{\mathbb R^{n+1}}
      \frac{|\Phi_{\varepsilon}\ast\delta V|^2}{\Phi_{\varepsilon}\ast\|V\|+\varepsilon}\,dx
      \leq \varepsilon^{\frac14} +\|u\|_{C^1}\|V\|(\mathbb R^{n+1}),
    \end{equation}
 and under the assumption of $\|f_\sharp V\|(\mathbb R^{n+1})\leq M$, \cite[(5.57)]{kim2017mean} becomes
 \begin{equation}\label{5.57}
     |\delta(V,\phi)(h_{\varepsilon}(\cdot, V)+u(\cdot))-\delta(f_{\sharp}V,\phi)(h_{\varepsilon}(\cdot,f_{\sharp}V)+u(\cdot))|
     \leq \varepsilon^{c_2-2n-19},
 \end{equation}
 and the same expression for \cite[(5.58)]{kim2017mean} with $\Omega=1$. Here $c_2:=3n+20$ (see \cite[(5.54)]{kim2017mean}). The proof of 
 \eqref{5.55} is the same, in essence because of the boundedness of $u$ in $C^1$ (to be precise, 
 the restriction on $\varepsilon$ also depends on $\|u\|_{C^1
}$). To obtain \eqref{5.56}, we may proceed as in \cite[(5.66)]{kim2017mean} with 
 \begin{align}\label{5.66}
     \delta(V,1)(h_\varepsilon+u)&=\delta V(h_\varepsilon)+\delta V(u)\nonumber\\
     &\leq -(1-\varepsilon^{\frac14})\int_{\mathbb R^{n+1}}
     \frac{|\Phi_{\varepsilon}\ast\delta V|^2}{\Phi_{\varepsilon}\ast\|V\|+\varepsilon}\,dx
     +\varepsilon^{\frac14}+\|u\|_{C^1}\|V\|(\mathbb R^{n+1}),
 \end{align}
 where we used \cite[(5.23)]{kim2017mean} and the 
 definition of $\delta V$. Combining \eqref{5.55}
 with $\phi\equiv 1$,
 \eqref{5.66} gives \eqref{5.56}. The same proof works for \eqref{5.57} and the one corresponding to \cite[(5.58)]{kim2017mean} by using the boundedness of $u$ in $C^1$. 
 The content of \cite[Section 6]{kim2017mean} needs only a
 slight change due to $u$, for example, in \cite[Proposition 6.1]{kim2017mean}, we have $\mathcal E_0:=\{E_{0,1}, \ldots,E_{0,N}\}\in \mathcal{OP}^N_1$, and there exists a family $\mathcal E_{j,l}\in\mathcal{OP}^N_1$ ($l=0,1,\ldots,j 2^{p_j}$) for each $j\in\mathbb N$ with
 \begin{equation}
     \mathcal E_{j,0}=\mathcal E_0 \mbox{ for all }j\in\mathbb N
 \end{equation}
 and with the notation of (see also \cite[(5.54)]{kim2017mean})
 \begin{equation}
     \Delta t_j:=\frac{1}{2^{p_j}}\in (2^{-1} \varepsilon_j^{c_2},\varepsilon_j^{c_2}],
 \end{equation}
 we have (cf.~\cite[(6.3)-(6.5)]{kim2017mean})
 \begin{equation}\label{6.3}
     \|\partial\mathcal E_{j,l}\|(\mathbb R^{n+1})
     \leq \mathcal H^1(\Gamma_0)\exp\Big(l\Delta t_j\|u\|_{C^1}\Big)+\varepsilon_j^{\frac18}l\Delta t_j,
 \end{equation}
 \begin{equation}\label{6.4}
 \begin{split}
     &\frac{\|\partial\mathcal E_{j,l}\|(\mathbb R^{n+1})-
     \|\partial\mathcal E_{j,l-1}\|(\mathbb R^{n+1})}{\Delta t_j}
     +\frac12
     \int_{\mathbb R^{n+1}}\frac{|\Phi_{\varepsilon_j}\ast
     \delta(\partial\mathcal E_{j,l})|^2}{\Phi_{\varepsilon_j}
     \ast\|\partial\mathcal E_{j,l}\|+\varepsilon_j}\,dx \\
&     -\frac{1-j^{-5}}{\Delta t_j}\Delta_j\|\partial\mathcal E_{j,l-1}\|(\mathbb R^{n+1})\leq \varepsilon_j^{\frac18}
     +\|u\|_{C^1}\|\partial\mathcal E_{j,l-1}\|(\mathbb R^{n+1}),
 \end{split}
 \end{equation}
 \begin{equation}\label{6.5}
     \frac{\|\partial\mathcal E_{j,l}\|(\phi)-\|\partial\mathcal E_{j,l-1}\|(\phi)}{\Delta t_j}
     \leq \delta(\partial\mathcal E_{j,l},\phi)(h_{\varepsilon_j}(\cdot,\partial\mathcal E_{j,l})+
     u(\cdot, l\Delta t_j))+\varepsilon_j^{\frac18}
 \end{equation}
 for $l=1,2,\ldots,j2^{p_j}$ and $\phi\in \mathcal A_j$. The proof of these is very similar. Starting
 with \cite[(6.6)]{kim2017mean}, we put
 \begin{equation}
     M_j:=\mathcal H^n(\Gamma_0)\exp\big(j\|u\|_{C^1}\big)+1,
 \end{equation}
 and choose $\varepsilon_j$ similarly so that 
 corresponding to this $M_j$, \cite[(6.7)]{kim2017mean} holds. Then choose an admissible 
 Lipschitz map $f_1\in {\bf E}^{vc}(\mathcal E_{j,l},j)$ (\cite[Definition 3.1]{ST-canonical}) as in \cite[(6.9)]{kim2017mean}, define $\mathcal E_{j,l+1}^*:=(f_1)_{\star}\mathcal E_{j,l}$ and define 
 \begin{equation}
     f_2(x):=x+\Delta t_j (h_{\varepsilon_j}(x,
     \partial\mathcal E_{j,l+1}^*)+u(x,(l+1)\Delta t_j)).
 \end{equation}
 Then we define 
 \begin{equation}
     \mathcal E_{j,l+1}:=(f_2)_{\star}\mathcal E_{j,l+1}^*.
 \end{equation}
 The rest of the proof is identical and one can deduce \eqref{6.3}-\eqref{6.5}. 
 The statement of \cite[Proposition 6.4]{kim2017mean} is the same in that there exists 
 a subsequence $\{j_l\}_{l\in\mathbb N}$ such that (see \cite[(6.18)-(6.20)]{kim2017mean})
 \begin{equation}
     \lim_{l\rightarrow\infty} \|\partial\mathcal E_{j_l}(t)\|(\phi)=\mu_t(\phi)
 \end{equation}
 for all $t\geq 0$ and for all $\phi\in C_c(\mathbb R^{n+1})$, and for all $T<\infty$
 \begin{equation}\label{lb1}
     \limsup_{i\rightarrow\infty} 
     \int_0^T\Big(\int_{\mathbb R^{n+1}}
     \frac{|\Phi_{\varepsilon_{j_l}} \ast\delta(\partial\mathcal E_{j_l}(t))|^2}{\Phi_{\varepsilon_{j_l}} \ast\|\partial\mathcal E_{j_l} (t)\|+
     \varepsilon_{j_l}}\,dx
     -\frac{1}{\Delta t_{j_l}} \Delta_{j_l}^{vc}\|\partial\mathcal E_{j_l}(t)\|(\mathbb R^{n+1})\Big)dt<\infty.
 \end{equation}
 Moreover, for a.e.~$t\in[0,\infty)$, we have
 \begin{equation}\label{lb2}
\lim_{l\rightarrow\infty}j_l^{2(n+1)}\Delta_{j_l}^{vc}\|\partial\mathcal E_{j_l}(t)\|(\mathbb R^{n+1})=0.
\end{equation}
The proofs of rectifiability and integrality 
\cite[Section 7, 8]{kim2017mean} (with 
slight modification as in \cite{ST-canonical}) utilize 
only the properties \eqref{lb1} and \eqref{lb2}, so that the measure $\mu_t$ 
is integral for a.e.~$t\geq 0$. The last part for the proof of 
\eqref{bra-n} is similar to that in \cite[Section 9]{kim2017mean} in that 
we have the extra term coming from $u$ in the right-hand side. Tracing the proof of \cite[Theorem 9.3]{kim2017mean}, it comes down (see \cite[(9.10)]{kim2017mean}) to analyzing the
behavior of (with $\partial\mathcal E_{j_l}=\partial\mathcal E_{j_l}(t)$)
\begin{equation}
    \delta(\partial\mathcal E_{j_l},\hat\phi)(h_{\varepsilon_{j_l}}+u_{j_l})
    =\delta(\partial\mathcal E_{j_l})(\hat\phi (h_{\varepsilon_{j_l}}+u_{j_l}))+
    \int_{{\bf G}_1(\mathbb R^{n+1})} S^\perp(\nabla\hat\phi)\cdot (h_{\varepsilon_{j_l}}+u_{j_l})\,d(\partial\mathcal E_{j_l})
\end{equation}
as $l\rightarrow\infty$ for a.e.~$t$. Here $u_{j_l}$
is a piece-wise constant (only in time-direction) vector field such that 
\begin{equation}
    u_{j_l}(x,t)=u(x,k\Delta t_{j_l})\mbox{ if }
    t\in ((k-1)\Delta t_{j_l},k\Delta t_{j_l}].
\end{equation}
Since $u\in C^1_c(\mathbb R^{n+1}\times[0,\infty);\mathbb R^{n+1})$, $u_{j_l}$ and $\nabla u_{j_l}$ converges
to $u$ and $\nabla u$ uniformly as $l\rightarrow\infty$. Now, the term $\delta(\partial\mathcal E_{j_l},\hat\phi)(u_{j_l})$
 is bounded uniformly (in terms of $\|u\|_{C^1}$ and 
 $\|\partial\mathcal E_{j_l}\|(\mathbb R^{n+1})$). Thus,
 in choosing a subsequence $\{j'_l\}_{j\in\mathbb N}\subset\{j_l\}_{l\in\mathbb N}$ 
 in \cite[(9.16)]{kim2017mean} such that
 \begin{equation}
   \limsup_{l\rightarrow\infty} \delta(\partial\mathcal E_{j_l}(t),\hat\phi)(h_{\varepsilon_{j_l}}+u_{j_l})=\lim_{l\rightarrow\infty} \delta(\partial\mathcal E_{j'_l}(t),\hat\phi)(h_{\varepsilon_{j'_l}}+u_{j'_l})  
 \end{equation}
 and proceeding similarly, one can prove that
 (\cite[(9.19)]{kim2017mean})
 \begin{equation}
     \limsup_{l\rightarrow\infty}\int_{\mathbb R^{n+1}}
     \frac{|\Phi_{\varepsilon_{j'_l}}\ast\delta(\partial\mathcal E_{j'_l}(t))|^2}{\Phi_{\varepsilon_{j'_l}}\ast \|\partial\mathcal E_{j'_l}(t)\|+\varepsilon_{j'_l}}\,dx\leq i\tilde M(t).
 \end{equation}
Here $\tilde M(t)$ is as defined in \cite[(9.18)]{kim2017mean} and is finite for a.e.~$t$. But for this
subsequence, $\{\partial\mathcal E_{j'_l}(t)\}_{l\in\mathbb N}$ converges to
$V_t\in {\bf IV}_n(\mathbb R^{n+1})$ as varifolds due to
\cite[Lemma 9.1(b)]{kim2017mean}, and one can 
conclude that
\begin{equation}
    \lim_{l\rightarrow\infty}\delta(\partial\mathcal E_{j'_l}(t),\hat \phi)(u_{j'_l})=\delta(V_t,\hat\phi)(u).
\end{equation}
The harder part on the convergence of $\delta(\partial\mathcal E_{j'_l},\hat\phi)(h_{\varepsilon_{j'_l}})$ is identical and one can
prove \eqref{bra-n} for a time-independent test function $\phi$, 
and one can proceed similarly for a time-dependent 
test function $\phi$ as well. 
This proves the existence of measures 
$\{V_t\}_{t\geq 0}$ satisfying (V1'), (V3') (for now
only $V_t\in{\bf IV}_n(\R^{n+2})$ a.e.$t$), (V4')-(V7'). 
As for (V2'), the same argument
in Remark \ref{re1} (or \cite[Proposition 6.10]{ST19}) shows
the claim, while (V3') in the case of $n=1$ 
is proved 
in \cite{Kim-Tone2}, where the only 
ingredients needed for the proof were \eqref{lb1}
and \eqref{lb2}. 

For $\{E_i(t)\}_{t\geq 0}$ we need modifications in \cite[Section 10]{kim2017mean} and \cite{ST-canonical}. For the first one, the main point is that
one needs to modify Huisken's monotonicity formula \cite{Huisken_mono} appropriately
due to the presence of $u$, which can be
done easily when $u$ is bounded. In
\cite[(10.2)]{kim2017mean}, for $t\in[0,T]$,
we replace 
$\hat{\rho}^R_{(y,s)}(x,t)$ by
\begin{equation}
    \hat\rho^R_{(y,s)}(x,t):=\eta\Big(\frac{x-y}{R}\Big)\exp\big(-t\|u\|_{L^\infty}^2\big)\rho_{(y,s)}(x,t).
\end{equation}
Then, proceeding as in \cite[Proposition 6.2]{kasai2014general} (see the computations up to \cite[(6.6)]{kasai2014general}), one can conclude
that for $0\leq t_1<t_2\leq T$,
\begin{equation}
\|V_t\|(\hat{\rho}^R_{(y,s)}(\cdot,t))\Big|_{t=t_1}^{t_2}
\leq c(n) R^{-2}(t_2-t_1)\sup_{t'\in[t_1,t_2]}
R^{-n}\|V_{t'}\|(B_{2R}(y)).
\end{equation}
which allows the same argument in \cite[Section 10]{kim2017mean}. Another minor modification is 
that of \cite[Lemma 10.12]{kim2017mean}.
With additional bounded $u$, one can prove:
\begin{lemma}
    For some $t\in \mathbb R^+$, $x\in \mathbb R^{n+1}$ and $r>0$, suppose $\|V_t\|(U_r(x))=0$. Then for $t'\in [t,t+\frac{r^2}{2n+2r\|u\|_{L^\infty}}]$, we have
    $\|V_{t'}\|(U_{\sqrt{r^2-(2n+2r\|u\|_{L^\infty})(t'-t)}}(x))=0$.
\end{lemma}
For completeness, we give the proof.
\begin{proof}
    Without loss of generality, assume $x=0$
    and $t=0$.
    Define $$\phi(x,t):=\{r^2-|x|^2-(2n+2r\|u\|_{L^\infty})t\}^4 \exp(-t\|u\|_{L^\infty}^2)$$
    for $|x|^2\leq r^2- (2n+2r\|u\|_{L^\infty})t$ and $=0$
    otherwise. Then $\phi$ is $C_c^2$ in 
    the space variables and by the
    Cauchy-Schwarz inequality and the first
    variation formula \eqref{tonegawakimequ2.2a}, we have
    \begin{equation}\label{int1}
        \int_{\mathbb R^{n+1}}
        (\nabla\phi-h\phi)\cdot(h+u^\perp)\,
        d\|V_t\|\leq \int_{G_n(\mathbb R^{n+1})}
        \frac12|u|^2\phi-S\cdot\nabla^2\phi+|u||\nabla\phi|\,dV_t(\cdot,S).
    \end{equation}
    Writing above $\phi$ as $R(x,t)^4\exp(-t\|u\|^2_{L^\infty})$, one can check that (also using $S\cdot I=n$)
    \begin{equation*}
        S\cdot \nabla^2\phi=(48 R^{-2}|S(x)|^2
        -8nR^{-1} )\phi
    \end{equation*}
    so that \eqref{int1} can be bounded from above
    by 
\begin{equation*}
  \int_{\mathbb R^{n+1}} \frac{\|u\|_{L^\infty}^2}{2}\phi+8n R^{-1}\phi
  +\|u\|_{L^\infty}8r R^{-1}\phi\,d\|V_t\|
\end{equation*}
while 
\begin{equation*}
    \frac{\partial \phi}{\partial t}
    =-\|u\|_{L^\infty}^2\phi-8(n+r\|u\|_{L^\infty})R^{-1}\phi.
\end{equation*}
Thus \eqref{bra-n} shows that $\|V_t\|(\phi)$ is decreasing in time
and we obtain the desired conclusion.
\end{proof} 
We can proceed exactly as in \cite[Section 10]{kim2017mean} and may conclude (E1')-(E6'). 
For (E7') and (E8'), we need to modify the proof in 
\cite{ST-canonical}. Since $u$ is $C^1$,
the proof is similar. In \cite[Section 4.1]{ST-canonical}, $h+u^\perp$ in place of
$h$ is a generalized velocity in the 
sense of $L^2$ flow and the proof is 
identical once \eqref{bra-n} is established.
The proof on the existence of measure-theoretic velocities (\cite[Section 4.2]{ST-canonical}) can be easily modified by replacing the motion law from $h_\varepsilon$ to $h_\varepsilon+u$ and 
the same proof shows \cite[Proposition 4.5]{ST-canonical} (note that $u_i$ in the 
statement there is not to be confused with $u$ in this paper). The proof for the identification of the measure-theoretic velocity in \cite[Section 4.3]{ST-canonical} goes verbatim and proves \eqref{gvel}. The rest of \cite{ST-canonical}, Section 5 and 6, goes without 
change and in particular, (E8') holds true for this modified version.


\begin{thebibliography}{10}

\bibitem{MR2424078}
{\sc R.~A. Adams and J.~J.~F. Fournier}, {\em Sobolev spaces}, vol.~140 of Pure
  and Applied Mathematics (Amsterdam), Elsevier/Academic Press, Amsterdam,
  second~ed., 2003.

\bibitem{allard1972first}
{\sc W.~K. Allard}, {\em On the first variation of a varifold}, Ann. Math., 95
  (1972), pp.~417--491.

\bibitem{brakke1978motion}
{\sc K.~A. Brakke}, {\em The Motion of a Surface by Its Mean Curvature},
  vol.~20 of Mathematical notes, Princeton University Press, Princeton, 1978.

\bibitem{Caffarelli:1982aa}
{\sc L.~Caffarelli, R.~Kohn, and L.~Nirenberg}, {\em Partial regularity of
  suitable weak solutions of the {N}avier-{S}tokes equations}, Comm. Pure Appl.
  Math., 35 (1982), pp.~771--831.

\bibitem{evans2015measure}
{\sc L.~C. Evans and R.~F. Gariepy}, {\em Measure theory and fine properties of
  functions}, Textbooks in Mathematics, CRC Press, revised~ed., 2015.

\bibitem{MR4645674}
{\sc S.~Hensel and Y.~Liu}, {\em The sharp interface limit of a
  {N}avier-{S}tokes/{A}llen-{C}ahn system with constant mobility: convergence
  rates by a relative energy approach}, SIAM J. Math. Anal., 55 (2023),
  pp.~4751--4787.

\bibitem{Huisken_mono}
{\sc G.~Huisken}, {\em Asymptotic behavior for singularities of the mean
  curvature flow}, J. Differential Geom., 31 (1990), pp.~285--299.

\bibitem{kasai2014general}
{\sc K.~Kasai and Y.~Tonegawa}, {\em A general regularity theory for weak mean
  curvature flow}, Calc. Var. Partial Differ. Equ., 50 (2014), pp.~1--68.

\bibitem{kim2017mean}
{\sc L.~Kim and Y.~Tonegawa}, {\em On the mean curvature flow of grain
  boundaries}, Ann. Inst. Fourier (Grenoble), 67 (2017), pp.~43--142.

\bibitem{Kim-Tone2}
\leavevmode\vrule height 2pt depth -1.6pt width 23pt, {\em Existence and
  regularity theorems of one-dimensional {B}rakke flows}, Interfaces Free
  Bound., 22 (2020), pp.~505--550.

\bibitem{LiuSato2010}
{\sc C.~Liu, N.~Sato, and Y.~Tonegawa}, {\em On the existence of mean curvature
  flow with transport term}, Interfaces Free Bound., 12 (2010), pp.~251--277.

\bibitem{liu2012two}
\leavevmode\vrule height 2pt depth -1.6pt width 23pt, {\em Two-phase flow
  problem coupled with mean curvature flow}, Interfaces Free Bound., 14 (2012),
  pp.~185--203.

\bibitem{LiuShen}
{\sc C.~Liu and J.~Shen}, {\em A phase field model for the mixture of two
  incompressible fluids and its approximation by a {F}ourier-spectral method},
  Phys. D, 179 (2003), pp.~211--228.

\bibitem{maggi2012sets}
{\sc F.~Maggi}, {\em Sets of finite perimeter and geometric variational
  problems: an introduction to Geometric Measure Theory}, vol.~135 of Cambridge
  Studies in Advanced Mathematics, Cambridge University Press, Cambridge, 2012.

\bibitem{meyers1977integral}
{\sc N.~G. Meyers and W.~P. Ziemer}, {\em Integral inequalities of
  {P}oincar{\'e} and {W}irtinger type for {BV} functions}, Amer. J. Math., 99
  (1977), pp.~1345--1360.

\bibitem{MR4158524}
{\sc M.~Pozzetta}, {\em A varifold perspective on the {$p$}-elastic energy of
  planar sets}, J. Convex Anal., 27 (2020), pp.~845--879.

\bibitem{simon1983lectures}
{\sc L.~Simon}, {\em Lectures on geometric measure theory}, vol.~3 of
  Proceedings of the Centre for Mathematical Analysis, Australian National
  University, Canberra, 1983.

\bibitem{ST19}
{\sc S.~Stuvard and Y.~Tonegawa}, {\em An existence theorem for {B}rakke flow
  with fixed boundary conditions}, Calc. Var. Partial Differ. Equ., 60 (2021),
  pp.~Paper No. 43, 53.

\bibitem{ST-canonical}
\leavevmode\vrule height 2pt depth -1.6pt width 23pt, {\em On the existence of
  canonical multi-phase {B}rakke flows}, Adv. Calc. Var., 17 (2024),
  pp.~33--78.

\bibitem{takasao2016existence}
{\sc K.~Takasao and Y.~Tonegawa}, {\em Existence and regularity of mean
  curvature flow with transport term in higher dimensions}, Math. Ann., 364
  (2016), pp.~857--935.

\bibitem{tonegawa2019brakke}
{\sc Y.~Tonegawa}, {\em Brakke's Mean Curvature Flow: An Introduction},
  SpringerBriefs in Mathematics, Springer, Singapore, 2019.

\end{thebibliography}
\end{document}